\documentclass[11pt, letterpaper, reqno]{amsart}

\usepackage{amssymb}
\usepackage{amsmath}
\usepackage{mathrsfs}
\usepackage{amsfonts}
\usepackage{amsthm}
\usepackage{graphicx}

\usepackage{subcaption}

\usepackage[colorlinks=true, pdfstartview=FitV, linkcolor=blue,citecolor=blue, urlcolor=blue]{hyperref}
\usepackage{verbatim}

\usepackage{pdfsync}

\usepackage{lmodern} 
\usepackage[english]{babel} 

\numberwithin{equation}{section}

\theoremstyle{definition}
\newtheorem{defi}{Definition}[section]
\newtheorem{rem}[defi]{Remark}
\theoremstyle{plain}
\newtheorem{theorem}{Theorem}[section]

\newtheorem{lem}[defi]{Lemma}
\newtheorem{cor}[defi]{Corollary}
\newtheorem{prop}[defi]{Proposition}

\numberwithin{equation}{section}

\def\R{\mathbb{R}}
\def\Z{\mathbb{Z}}

\def\C{\mathbb{C}}

\def\PP{\mathcal{P}}
\def\RR{\mathcal{R}}
\def\SS{\mathcal{S}}

\renewcommand{\r}{\mathbb{R}}

\newcommand{\z}{\mathbb{Z}}

\renewcommand{\and}{\quad\textrm{ and }\quad}

\DeclareMathOperator{\curl}{curl}

\let \epsilon\varepsilon
\let \phi\varphi

\author{Tomasz Cie\'{s}lak}
\address{Institute of Mathematics \newline Polish Academy of Sciences \newline \'Sniadeckich 8, 00-656 Warszawa, Poland}
\email{cieslak@impan.pl}

\author{Piotr Kokocki}
\address{Faculty of Mathematics and Computer Science \newline Nicolaus Copernicus University \newline Fryderyka Chopina 12/18, 87-100 Toru\'n, Poland}
\email{pkokocki@mat.umk.pl}

\author{Przemys\l{}aw Kosewski}
\address{Faculty of Mathematics and Information Sciences \newline Warsaw University of Technology \newline
Koszykowa 75, 00-662 Warsaw, Poland}
\email{kosewski.przemek@gmail.com}

\title[Uniqueness problem for Prandtl spirals]{Uniqueness problem for Prandtl spirals}

\begin{document}
\begin{abstract}
This paper is devoted to the study of the following problem: Is the divergence-free velocity field introduced in \cite{CKO} the unique such field associated with vorticity given by the Prandtl spiral? Without additional assumptions, the answer is negative. However, we show that if the class of admissible velocities is restricted to those satisfying the velocity matching condition and an appropriate decay condition at the origin of the spiral, then the velocity field is uniquely determined. We subsequently extend the result to the case of fields with vorticity composed of unions of concentric logarithmic spirals. As a by-product, we derive an alternative way of deriving formula for the velocity corresponding to the Prandtl spirals.

The proof relies on an approach that is of independent interest. We construct an explicit conformal map from the exterior of a logarithmic spiral onto a strip. This transformation reduces the problem to establishing the uniqueness of a holomorphic function defined on the strip, under non-standard boundary and decay conditions.
\end{abstract}



\maketitle

\section{Introduction}
The velocity field corresponding to the vorticity induced by a single Prandtl spiral, as well as by a set of $N\geq 2$ uniformly distributed such spirals, was explicitly constructed in \cite{CKO}. Recall that the Prandtl spirals are well-known objects in hydromechanics, postulated by Helmholtz himself in his foundational paper on discontinuous flows \cite{Helmholtz}. They were considered due to the observability of spiraling flows in nature and the conjecture -- tracing back to Helmholtz -- that such structures are inherently unstable and may give rise to further instabilities, such as the well-known von K\'arm\'an vortices. Many authors have investigated flows exhibiting spiral patterns -- see, for instance, Hamel \cite{Hamel}. A well-known example of logarithmic spiral vortex sheets was introduced by Prandtl \cite{Prandtl}, who associated them with the flow detachment from the wingtip of a plane during takeoff. Alexander later extended Prandtl’s construction to several uniformly distributed spirals; see \cite{alexander}. For more recent observations of spiral-patterned flows, see \cite{Everson}; for applications in aviation, see \cite{saffman}.

While Prandtl spirals have appeared in the engineering literature for over a century, their rigorous and exact mathematical status was not well understood until very recently. First, Elling and Gnann \cite{Elling} observed that a family of $N\geq 3$ uniformly distributed logarithmic spirals satisfies the Birkhoff-Rott equation. Subsequently, the first two authors of the present article, together with O\.za\'nski, showed that not only a single Prandtl spiral, but also any set of $N\geq 2$ uniformly distributed logarithmic spirals, serves a support of the vorticity field that gives rise to velocity field solving the weak 2D Euler equation, see \cite{CKO}. In a moment, we shall review the main points of this construction. Let us only recall that bifurcation branches of non-symmetric logarithmic spirals solving the 2D Euler equation was constructed in \cite{CKO2}. An alternative construction of the Prandtl spiral as a weak solution to the 2D Euler equation was provided by Jeong and Said in \cite{JS}, where the results of \cite{CKO2} were also independently rederived using a different approach.

Let us recall that the Prandtl spiral is a time-dependent flow with a vorticity concentrated on a family of curves $\{\Sigma(t)\}_{t>0}$, parameterized by
\begin{equation}\label{spirala}
Z(\theta,t) =t^\mu e^{a\theta}e^{i\theta}, \quad t>0,
\end{equation}
where $\theta\in\R$ is an angular parameter and $a>0$, $\mu\in\R$ are constants. Moreover, the total circulation along the spiral is given by
\begin{equation*}
\Gamma(\theta, t)=gt^{2\mu-1}e^{2a\theta}, \quad t>0,
\end{equation*}
where $g\in\r\setminus\{0\}$ is a constant. In \cite[eq. (1.12)]{CKO}, the following formula for the velocity field associated with the Prandtl spiral was derived:
\begin{equation}\label{predkosc}
v(z,t)=t^\mu w(z/t^\mu), \quad z\in \RR(t), \ t>0,
\end{equation}
where $\RR(t):=\C\setminus(\Sigma(t)\cup\{0\})$ is the time-dependent complement of the spiral and the profile function is expressed in the polar coordinates as 
\begin{equation}\label{form-w}
w(z)=e^{i\theta}\frac{2ag}{r(a-i)}\left(r^{\frac{2a}{a+i}}e^{-A\theta}\frac{e^{2\pi J(r,\theta)A}}{1-e^{2\pi A}}\right)^*
\end{equation} 
for $z=re^{i\theta}\in \RR:=\RR(1)$. In the above formula $A:=-2ai/(a+i)$ and $J(r,\theta)$ is a winding number of the spiral $\Sigma:=\Sigma(1)$, given by 
\begin{equation}\label{def-j}
J(r,\theta)=\min \{j\in \z \ | \ a(2\pi j-\theta)+\ln r >0\}.
\end{equation}
We know that the profile function \eqref{form-w} can be recovered as the derivative of a suitable complex potential. In fact, one can show that $w^{*} = \Phi'$, where 
\begin{align}\label{eq-phi}
\Phi(z) = \frac{g}{1-e^{2\pi A}}e^{iA(\ln r + i(\theta - 2\pi J(r,\theta)))}, 
\end{align}
for $z=re^{i\theta}\in\RR$. In \cite[Theorem 1.8]{CKO}, it was shown that $v$ is a weakly divergence-free vector field satisfying
\begin{equation}\label{eq-omega-1} 
\mathrm{curl}\,v (t) = \gamma(t)\, \delta_{\Sigma(t)}, \quad t>0
\end{equation}
in the distributional sense, where the measure density is given by
\begin{equation}\label{gestosc2bb}
\gamma(t,Z(t,\theta)):=\frac{\partial_\theta \Gamma(t,\theta)}{|\partial_\theta Z(t,\theta)|} =
\frac{2ag t^{\mu-1}}{\sqrt{1+a^{2}}} e^{a\theta}, \quad \theta\in\R.
\end{equation}
Moreover, the velocity profile satisfies the decay condition at the origin
\begin{equation}\label{decay}
|w(z)| \lesssim |z|, \quad z\to 0, \ z\in\RR.
\end{equation}
A natural question then is whether the time dependent velocity field defined by \eqref{predkosc}, \eqref{form-w} is the only one on $\R^{2}$ whose vorticity matches \eqref{eq-omega-1} and \eqref{gestosc2bb}, and whose decay at the origin satisfies \eqref{decay}. This question is both significant and nontrivial for the following reasons. \\[3pt]
{\em Irregularity of the Prandtl-spiral vorticity.} 
The vorticity associated with the Prandtl spiral is so singular that the Biot–Savart integral formula fails to converge, preventing direct use of the usual integral representation. \\[3pt]
{\em Non-uniqueness up to holomorphic additions.}
Away from the spiral itself, the flow is irrotational, so adding any holomorphic function to $w$ leaves the vorticity unchanged. Therefore, additional constraints on the class of velocity fields must be imposed to rule out such ``hidden'' degrees of freedom. \\[3pt]
{\em Linear growth at infinity.}
Our method does not impose any growth‐type restrictions on $w$. Since the velocity field grows linearly at infinity, no Liouville type argument (which would require boundedness) can be applied. \\[3pt]
\indent Following \cite{CKO}, we focus on the self-similar logarithmic spiral vortex sheets, for which the profile map $w$ satisfies the continuity condition for normal components
\begin{equation}\label{skoki-nor}
(w^R(z)-w^L(z))\cdot \vec{n}(z) = 0, \quad z\in \Sigma,
\end{equation} 
and the tangential jump condition 
\begin{equation}\label{skoki-tan}
(w^R(z)-w^L(z))\cdot\vec{\tau}(z) = \gamma(z), \quad z\in \Sigma, 
\end{equation}
where $w^{R}(z)$ and $w^{L}(z)$ are the limit velocities $w(z')$ as $z'\to z$ from the right and left sides of the spiral $\Sigma$ with respect to its natural orientation. Moreover, $\vec{\tau}(z)$ and $\vec{n}(z) := -i\vec{\tau}(z)$ are the unit tangent and normal vectors to the spiral, respectively, and $\gamma:=\gamma(1)$ is a density function. According to the formula \eqref{spirala}, the spiral $\Sigma$ is parametrized by the mapping
\begin{equation*}
Z(\theta):=Z(\theta,1) = e^{a\theta}e^{i\theta}, \quad \theta \in \mathbb{R}
\end{equation*}
and the corresponding unit tangent and normal vectors are given by
\begin{align}\label{tan-form}
\vec{\tau}(Z(\theta)) & = \frac{1}{\sqrt{a^2+1}} [a \cos \theta - \sin \theta, \cos \theta + a \sin \theta], \\ \label{norm-form}
\vec{n}(Z(\theta)) & = \frac{1}{\sqrt{a^2+1}} [\cos \theta + a \sin \theta, \sin \theta - a \cos \theta].
\end{align}
Furthermore, the density function satisfies
\begin{equation}\label{gestosc}
\gamma(Z(\theta))=\frac{\partial_\theta \Gamma(\theta)}{|\partial_\theta Z(\theta)|} =
\frac{2ag}{\sqrt{1+a^{2}}} e^{a\theta}, \quad \theta\in\R,
\end{equation}
where $\Gamma(\theta) := \Gamma(\theta, 1)$. Let us emphasize that \eqref{skoki-nor} and \eqref{skoki-tan} represent the standard jump conditions for vortex sheets. They are, respectively, equivalent to the condition that the velocity profile $w$ is divergence-free and that its vorticity satisfies \eqref{eq-omega-1} in the sense of distributions. 

Moreover, as we point out in \cite[Theorem 1.2]{CKO}, if the self-similar vector field $v$ is a weak solution of the 2D Euler equation, then the profile function satisfies the velocity matching condition
\begin{equation}\label{vel_match}
\vec{n}(z)\cdot\left(w(z)-\mu z\right)=0, \quad z\in \Sigma.
\end{equation}
Observe that condition \eqref{vel_match} is well-defined since, by \eqref{skoki-nor}, we can set $\vec{n}(z)\cdot w(z) := \vec{n}(z)\cdot w^{L}(z) = \vec{n}(z)\cdot w^{R}(z)$ for $z\in\Sigma$. We are now ready to formulate our main theorem and discuss the strategy of its proof. 
\begin{theorem}\label{glowne}
Let $w$ be a two-dimensional weakly divergence-free velocity field satisfying conditions \eqref{vel_match} and \eqref{decay}. Furthermore, suppose that
\begin{equation}\label{dist-eq}
\curl w = \gamma\, \delta_{\Sigma}
\end{equation}
in the sense of distributions, where $\delta_{\Sigma}$ denotes the Dirac measure supported on the spiral $\Sigma$. Then $w$ is uniquely determined. Moreover, if we additionally assume that the spiral parameters satisfy the condition
\begin{equation}\label{im_Pr}
2ag \sin\left(\frac{4\pi a^2}{1+a^2}\right) = \mu \left(2\cos\left(\frac{4\pi a^2}{1+a^2}\right) - e^{\frac{-4\pi a}{1+a^2}} - e^{\frac{4\pi a}{1+a^2}}\right),
\end{equation}
then the field $w$ is given by the formula \eqref{form-w}.
\end{theorem}
The condition \eqref{im_Pr} is precisely the pressure matching condition (see \cite{CKO}), which is necessary for a logarithmic spiral vortex sheet to be a weak solution of the 2D Euler equation. When this holds, the velocity field $w$ is given by \eqref{form-w} and it is a weak solution to the 2D Euler equation.
We also recall that assuming $w$ to be weakly divergence-free and that its curl in the weak sense equals $\gamma\,\delta_{\Sigma}$, the conditions \eqref{skoki-nor} and \eqref{skoki-tan} follow immediately. 

Observe that Theorem \ref{glowne} establishes the existence of at most one velocity field corresponding to the Prandtl spiral vorticity, provided \eqref{vel_match} and \eqref{decay} hold. These are necessary conditions for a spiral vortex sheet to solve the 2D Euler equation.
As we will see in the proof of Theorem \ref{glowne}, unless the pressure matching condition is also assumed, the velocity $w$ does not exist.

Let us furthermore emphasize the role of the velocity matching condition in the above theorem. Indeed, without this assumption, the theorem no longer holds, as illustrated in the following remark.
\begin{rem}\label{poglowne}
Let $w$ be a 2D weakly divergence-free velocity field satisfying the decay condition \eqref{decay} and the distributional equality \eqref{dist-eq}. Then the perturbed mapping $w(z)+f(z)$, where $f$ is any holomorphic function satisfying $f(0)=0$, also satisfies \eqref{decay}. Moreover, it is divergence-free, fulfills the conditions \eqref{skoki-nor} and \eqref{skoki-tan}, and its $\curl$ in the distributional sense remains equal to $\gamma\delta_{\Sigma}$. \hfill $\square$
\end{rem}
A further results obtained in this paper concerns the case of multiple concentric spirals. We prove the uniqueness of the divergence-free velocity field corresponding to vorticity supported on a family of such spirals. As in the case of a single Prandtl spiral, uniqueness follows from the velocity matching condition and the decay at the origin. The precise statement is given in Section~\ref{wiele_spiral}, where we formulate the non-standard boundary value problem that characterizes the divergence-free velocity field corresponding to vorticity supported on a family of concentric logarithmic spirals. Under the additional assumption of the pressure matching condition, the boundary problem also yields an alternative derivation of the corresponding velocity formula; see Theorem~\ref{th_wiele_spiral}. In this case, the pressure matching condition takes the form of a system of discrete equations discrete equations involving the parameters that describe the spirals.

Let us briefly outline the proof of Theorem~\ref{glowne}, which will be presented in several steps. We begin by deriving an explicit formula for a conformal mapping that sends an infinite vertical strip in the complex plane onto the exterior $\RR$ of the logarithmic spiral $\Sigma$. This mapping is of independent interest and, up to our knowledge, it does not appear in the existing literature (see, for example \cite{Ahlfors} or \cite{Rudin}). A detailed analysis of this transformation and its interesting connection with the winding number \eqref{def-j} is given in Section~\ref{przeksztalcenie} (see also Lemma \ref{lem-winding}). Let us also mention that logarithmic spirals appear in other areas of mathematics; for instance, they play a key role in \cite{Markovic}, where they are used to construct counterexamples to conjectures of Thurston and Sullivan. We hope that our conformal mapping will find applications beyond hydromechanics. 

In Section~\ref{jedna_spirala}, we present the proof of the uniqueness part of Theorem~\ref{glowne}. To this end, we transform the setting to a strip via the conformal mapping from Section \ref{przeksztalcenie}. Although the reduction is crucial it does not lead to a standard problem. We shall deal with the uniqueness of the holomorphic function in a strip, but the boundary values on both edges of the strip are not fully known. Instead, we shall face a rather non-standard problem, where only the differences in the imaginary part of the function between the two sides of the strip are known. This follows from the assumption \eqref{skoki-tan}, which corresponds to the fact that the curl of the considered velocity is a measure supported on the Prandtl spiral. Let us note that the decay at the boundary of the strip suggests one could bypass the non-standard problem by applying a Phragm\'en-Lindel\"{o}f type principle to identify the velocity. In this context, the results of Widder \cite{Widder} would be particularly useful. However, this approach comes at the cost of restricting the class of admissible functions to those exhibiting sufficient decay at infinity. As we shall see in Section~\ref{sec-4}, this would require significantly limiting the range of parameters $a>0$ under consideration.  

Then, in Section~\ref{sec-velocity}, we present an alternative derivation of the formulas \eqref{predkosc} and \eqref{form-w} for the velocity field~$v$ using the conformal transformation. These formulas play an important role in understanding Prandtl spirals as weak solutions to the 2D Euler equations. They actually show that the asymptotic behavior of the velocity field near the origin of the spiral differs from what is commonly assumed in the aerodynamic literature; see the discussion at the bottom of p. 519 in \cite{CKO} concerning the asymptotic expansion from \cite{saffman}. Finally, the extension of our result to a family of concentric logarithmic spirals is presented in Sections~\ref{wiele_spiral}\,--\,\ref{sec-velocity-m}. Let us mention that our present derivation of the velocity formula does not make use of the tricky cancellation relations in \cite{Elling}.

\section{Conformal mapping}\label{przeksztalcenie}
In this section, we provide a formula for the conformal mapping from a strip onto the exterior $\RR$ of logarithmic spiral $\Sigma$. We then study its properties, which will be useful both in our proof of uniqueness and in an alternative computation of the velocity field associated with the Prandtl spiral. We also recall the definition of the winding number \eqref{def-j}, which plays a crucial role in the construction of the inverse conformal mapping and is also essential for the velocity formula derived in \cite{CKO}.
\begin{lem}\label{mapping}
Let us assume that 
\[
\SS:=\left\{z\in\C \,:\,  -\frac{2\pi a}{1+a^2}<\mathrm{Re}\,z<0\right\}
\]
is a vertical strip in the complex plane. Then the following function 
\begin{equation*}
f(z):=e^{(1-ai)z}, \quad z\in \SS
\end{equation*}
is a well-defined biholomorphic mapping between $\SS$ and $\RR$ with inverse
\begin{align*}
f^{-1}(z) = \frac{\ln r - a\theta + 2a\pi(J(r,\theta) - 1)}{1 + a^{2}} + i\frac{\theta + a\ln r - 2\pi(J(r,\theta) - 1) }{1 + a^{2}}
\end{align*}
for $z= re^{i\theta}\in \RR$, where $r>0$ and $\theta\in \r$. We recall that $J(r,\theta)$ is defined in the formula \eqref{def-j} and 
$\RR=\C\setminus(\Sigma\cup\{0\})$ denotes the complement in the complex plane of the logarithmic spiral. In particular, $f$ is conformal.
\end{lem}
\begin{proof}
To check that the map $f$ takes values in the set $\mathcal{R}$, we suppose on the contrary that $f(z)\in\Sigma$ for some $z = x+iy\in \SS$. Then we have
\begin{align*}
f(z) = e^{x+ay} e^{i(y-ax)} = e^{a\theta}e^{i\theta}
\end{align*}
for some $\theta\in\R$. Hence $x+ay = a\theta$ and $\theta = y-ax - 2k\pi$ for some $k\in\Z$, which in turn implies that
$x+ay = a\theta = a(y-ax - 2k\pi)$. This gives $(1+a^{2})x = - 2ak\pi$, which contradicts the fact that $x+iy\in\SS$.

To show that the map $f$ is injective, we take $z_{1} = x_{1} + iy_{1},z_{2} = x_{2} + iy_{2}\in \SS$ such that
$f(z_{1}) = f(z_{2})$. Then we have
\begin{align*}
e^{x_{1}+ay_{1}} e^{i(y_{1}-ax_{1})} = e^{x_{2}+ay_{2}} e^{i(y_{2}-ax_{2})}
\end{align*}
and consequently $e^{x+ay} e^{i(y-ax)} = 1$, where $x:=x_{1}-x_{2}$ and $y:=y_{1}-y_{2}$. This implies that $x+ay = 0$ and $y-ax = 2k\pi$ for some $k\in\Z$, which in turn gives $x + a(ax+2k\pi) = 0$ and $x=-\tfrac{2k\pi a}{1+a^{2}}$.
On the other hand $$x=x_{1}-x_{2}\in\left(-\tfrac{2\pi a}{1+a^{2}}, \tfrac{2\pi a}{1+a^{2}}\right),$$ which implies that $k=0$. Hence $x=y=0$ and the map $f$ is injective as desired.

To verify that $f$ is a surjective map, we take $z = re^{i\theta}\in\mathcal{R}$, where $r>0$ and $\theta\in\R$, and we are looking for $z=x+iy\in \SS$ such that $f(z) = re^{i\theta}$.
To this end we solve the system of equations
$$x+ay = \ln r, \quad y-ax = \theta - 2\pi(J(r,\theta) - 1),$$ (the factor $J(r,\theta)-1$ is such in order to make sure that the obtained $x\in (-\frac{2\pi a}{1+a^2},0)$, see below) and obtain
\begin{align*}
x =  \frac{\ln r - a(\theta - 2\pi(J(r,\theta) - 1))}{1 + a^{2}}, \quad y = \frac{(\theta - 2\pi(J(r,\theta) - 1)) + a\ln r}{1 + a^{2}}.
\end{align*}
From the definition \eqref{def-j} it follows that $z=x+iy\in \SS$, and hence $f$ is surjective as desired.
\end{proof}
The formula for an inverse $f^{-1}$ in the above lemma shows that $y\rightarrow -\infty$ in the strip frame is equivalent to $z\rightarrow 0$ in the spiral reference frame. 
\begin{cor}
The mapping $f$ satisfies the following limit
\begin{align}\label{lim-f1}
f(x+iy)\to 0 \quad \text{as}  \ \ y\to-\infty,
\end{align}
uniformly with respect to $-\tfrac{2\pi a}{1+a^2}<x<0$. Moreover, we have
\begin{align*}
\mathrm{Im}\, f^{-1}(z)\to -\infty, \quad z\to 0, \ z\in\RR.
\end{align*}
\end{cor}
\begin{proof}
The limit \eqref{lim-f1} follows from definition of the function $f$. To verify the latter, note that from the definition \eqref{def-j}, it follows that
\begin{align*}
2\pi J(r,\theta)  - \theta > -a^{-1}\ln r, \quad r>0, \ \theta\in\R.
\end{align*}
Hence, by Lemma \ref{mapping}, the following holds
\begin{align*}
\mathrm{Im}\, f^{-1}(r,\theta) & = \frac{\theta + a\ln r - 2\pi(J(r,\theta) - 1) }{1 + a^{2}} \\
& < \frac{(a+a^{-1})\ln r + 2\pi }{1 + a^{2}} \to -\infty \quad\text{when} \ \ r\to 0,
\end{align*}
which yields the desired result.
\end{proof}

\begin{rem}\label{pomapping}
Note that the mapping $f$ takes the two vertical lines 
$$\left\{\mathrm{Re}\,z = 0\right\} \quad\text{and}\quad \left\{\mathrm{Re}\,z = -2\pi a/(1+a^2)\right\},$$
which together bound the strip $\SS$, onto the logarithmic spiral $\Sigma$. Moreover, it satisfies the periodicity relation (see Figure \ref{fig:10a})
\begin{align*}
f( z_{1}) = f( z_{2}), \quad  z_{1} =  z_{2} -\frac{2\pi a}{1+a^{2}} + \frac{2\pi i}{1+a^2}, \quad  z_{2}\in i\R.
\end{align*}
Moreover, the point on the spiral is approached from two different sides by different parts of the boundary of the strip $\SS$. \hfill $\square$
\end{rem}

\begin{figure}[h]
    \centering
    \begin{subfigure}{0.45\textwidth}
        \includegraphics[width=\linewidth]{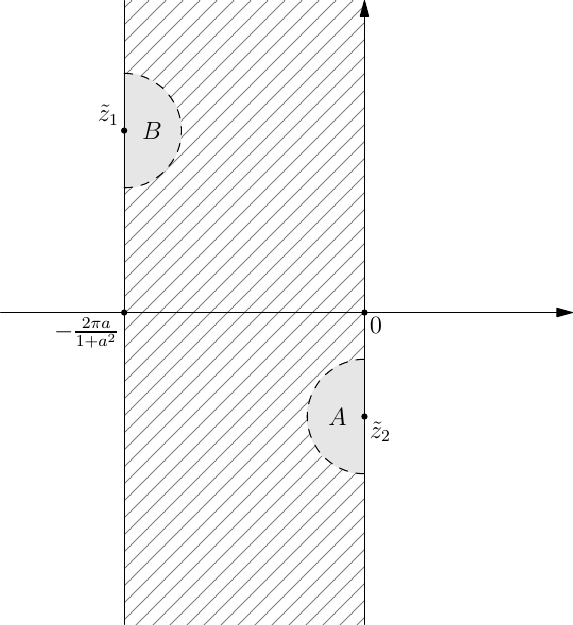}
    \end{subfigure}
    \hfill
    \begin{subfigure}{0.45\textwidth}
        \includegraphics[width=\linewidth]{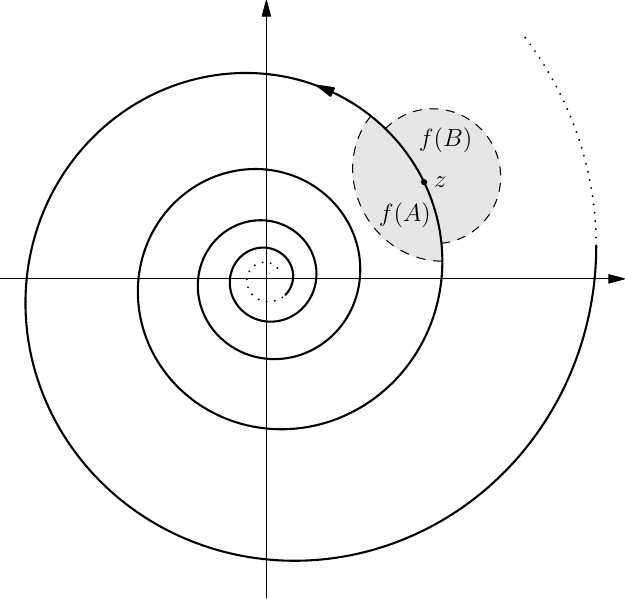}
    \end{subfigure}
\caption{Left: The points $\tilde z_{1} = \tilde z_{2} -\frac{2\pi a}{1+a^{2}} + \frac{2\pi i}{1+a^2}$, $\tilde z_{2}\in i\R$ and the corresponding regions $A$, $B$ located inside the strip $\SS$. Right: Logarithmic spiral with its natural orientation and the images of the regions $A$ and $B$ under the map $f$. Moreover $z=f(\tilde z_{1}) = f(\tilde z_{2})$.}
\label{fig:10a}
\end{figure}

\section{Uniqueness proof for a single spiral}\label{jedna_spirala}
In this section we shall use the holomorphic functions theory to obtain the proof of the uniqueness in Theorem \ref{glowne}. After applying Lemma \ref{mapping}, the problem is reduced to the problem of uniqueness of holomorphic functions on the strip, but a non-standard one. We do not know the full boundary conditions. Moreover, we do not need to assume any decay at an upper infinity of the strip to obtain our result. Instead, we shall carefully analyze the problem using the Schwarz reflection principle and arrive at the uniqueness in a less standard way.

The following straightforward consequence of the Schwarz reflection principle and the identity lemma for holomorphic functions will be useful in what follows.
\begin{prop}\label{holojedno}
Suppose $h$ is a holomorphic function defined on the open vertical strip
\[
S_{x_{0},x_{1}} := \{x + it \,:\, x \in (x_0, x_1), \, t \in \R\},
\]
where $x_0 < x_1$, and assume that $f$ extends continuously to the boundary of $S$. If $h(x_1 + it) = 0$ for all $t \in \R$, then $h \equiv 0$ on the entire strip $S_{x_{0},x_{1}}$.
\end{prop}
\begin{proof}
Since $h$ vanishes on the right edge of the strip, i.e., on the line $\{x_1 + it : t \in \R\}$, the Schwarz reflection principle (see \cite[Theorem 11.14]{Rudin}) implies that $f$ can be extended to a holomorphic function on the larger region
\[
\widetilde{S}_{x_{0},x_{1}} := \{x + it \,:\, x \in (x_0, 2x_1 - x_0), \, t \in \R\}.
\]
This extension agrees with $h$ on the original strip and remains holomorphic throughout the extended domain. Since the extended function is holomorphic and vanishes on the nontrivial vertical line $\{x_1 + it : t \in \R\}$ contained in $\widetilde{S}$, the identity theorem (see \cite[Corollary, p.~209]{Rudin}) implies that the extension, and hence $h$ itself, must vanish identically on $S_{x_{0},x_{1}}$.
\end{proof}

\subsection{Velocity field and boundary conditions on the strip}\label{sec-frame}
Let us assume that the velocity field $w$ and its holomorphic potential are given in the spiral reference frame and are connected by the 
relationship $w^{*}(z) = \Phi'(z)$ for $z\in\RR$. By applying the conformal mapping $f$ we obtain the corresponding functions $\tilde w(z)$ and $\tilde\Phi(z)$ defined in the strip reference frame, where 
\begin{equation}\label{transform}
\tilde w^{*}(z) = \tilde \Phi'(z) \quad\text{and}\quad  \tilde\Phi(z) = \Phi(f(z)), \quad z\in\SS.
\end{equation}
Writing the potentials as $\Phi = \Phi_{1} + i\Phi_{2}$ and $\tilde\Phi = \tilde \Phi_{1} + i\tilde \Phi_{2}$, we can calculate the velocity field in the strip $\tilde w^{*} = \tilde w_1 - i\tilde w_2 = \tilde \Phi'$, as
\begin{equation}\label{v1}
\tilde w_{1}(z) = \frac{\partial \tilde\Phi_{1}}{\partial x}(z) = \frac{\partial\Phi_{1}}{\partial x}(f(z)) \frac{\partial}{\partial x}\mathrm{Re}\,f(z) + \frac{\partial \Phi_{1}}{\partial y}(f(z)) \frac{\partial}{\partial x}\mathrm{Im}\,f(z)
\end{equation}
and
\begin{equation}\label{v2}
\tilde w_{2}(z) = \frac{\partial \tilde\Phi_{1}}{\partial y}(z) = \frac{\partial \Phi_{1}}{\partial x}(f(z)) \frac{\partial}{\partial y}\mathrm{Re}\,f(z) + \frac{\partial \Phi_{1}}{\partial y}(f(z)) \frac{\partial}{\partial y}\mathrm{Im}\,f(z).
\end{equation}
Let us observe that by the Cauchy-Riemann equations we have
\begin{equation}
\begin{aligned}\label{eq:Re_du}
\frac{\partial}{\partial y} \mathrm{Im}\,f(z) = \frac{\partial}{\partial x} \mathrm{Re}\,f(z) & = \frac{\partial}{\partial x} [e^{x+ay} \cos(y - ax)] \\
&= e^{x+ay} (\cos(y - ax) + a \sin(y - ax)),
\end{aligned}
\end{equation}
and
\begin{equation}
\begin{aligned}\label{eq:Im_du}
-\frac{\partial}{\partial y} \mathrm{Re}\,f(z) = \frac{\partial}{\partial x} \mathrm{Im}\,f(z) & = \frac{\partial}{\partial x} [e^{x+ay} \sin(y - ax)] \\
&= e^{x+ay} (\sin(y - ax) - a \cos(y - ax)).
\end{aligned}
\end{equation}
Finally, applying \eqref{eq:Re_du} and \eqref{eq:Im_du} in \eqref{v1}, we obtain
\begin{equation}\label{vpierwsza}
\begin{aligned} 
\tilde{w}_1(z)&= w_1(f(z))e^{x+ay}(\cos(y - ax) +a \sin(y - ax))\\
& \quad + w_2(f(z))e^{x+ay} (\sin(y - ax) - a \cos(y - ax)),
\end{aligned}
\end{equation}
while plugging \eqref{eq:Re_du} and \eqref{eq:Im_du} in \eqref{v2}, we arrive at
\begin{equation}\label{vdruga}
\begin{aligned}
\tilde{w}_2(z) &= w_1(f(z))e^{x+ay}\left(a\cos(y-ax)-\sin(y-ax)\right) \\
&\quad +w_2(f(z))e^{x+ay}\left(a\sin(y-ax)+\cos(y-ax)\right).
\end{aligned}
\end{equation}
We have already expressed the velocity in the strip in terms of the velocity in the spiral’s reference frame. Now let us examine how the boundary and decay conditions transform. Combining \eqref{vpierwsza}, \eqref{vdruga}, \eqref{tan-form}, \eqref{norm-form}, we have
\begin{align}\label{eq-bound-1a}
\tilde{w}_{1}(i y) & = e^{ay} \sqrt{a^2+1}  w^{L}(f(iy)) \cdot \vec{n}(f(iy))
\end{align}
and
\begin{align}\label{eq-bound-1b}
\tilde{w}_{2}(i y) & = e^{ay} \sqrt{a^2+1}  w^{L}(f(iy)) \cdot \vec{\tau}(f(iy))
\end{align}
for $y\in\R$. On the other hand, applying \eqref{vpierwsza}, \eqref{tan-form} and \eqref{norm-form} once again, together with the periodicity of $f$ (see Remark \ref{pomapping}), yields
\begin{equation}\label{eq-bound-2a}
\tilde{w}_{1}\left(-\frac{2a\pi}{1+a^2} + i\left(y + \frac{2\pi}{1+a^2}\right)\right) = e^{ay} \sqrt{a^2+1}  w^{R}(f(iy)) \cdot \vec{n}(f(iy))
\end{equation}
and
\begin{equation}\label{eq-bound-2b}
\tilde{w}_{2}\left(-\frac{2a\pi}{1+a^2} + i\left(y + \frac{2\pi}{1+a^2}\right)\right) = e^{ay} \sqrt{a^2+1}  w^{R}(f(iy)) \cdot \vec{\tau}(f(iy))
\end{equation}
for $y\in\R$.
Combining the equations \eqref{skoki-nor}, \eqref{norm-form}, \eqref{vel_match}, and \eqref{eq-bound-1a} we obtain
\begin{equation}
\begin{aligned}\label{eq-36aa}
\tilde{w}_{1}(i y) & = e^{ay} \sqrt{a^2+1}  w^{L}(f(iy)) \cdot \vec{n}(f(iy)) \\
&= e^{ay} \sqrt{a^2+1} \mu e^{(a+i)y} \cdot \vec{n}(f(iy)) \\
&= \mu e^{2ay} (\cos y(\cos y + a\sin y) + \sin y(\sin y - a\cos y)) \\
&= \mu e^{2ay} \quad \text{for} \  y\in\R.
\end{aligned}
\end{equation}
The above equations are not surprising as the formula could be guessed due to the conformality of the map $f$. Similarly, applying the equations \eqref{skoki-nor}, \eqref{norm-form}, \eqref{vel_match}, and \eqref{eq-bound-2a}, we have
\begin{equation}
\begin{aligned}\label{eq-36bb}
& \tilde{w}_{1}\left(-\frac{2a\pi}{1+a^2} + i\left(y + \frac{2\pi}{1+a^2}\right)\right) = e^{ay} \sqrt{a^2+1}  w^{R}(f(iy)) \cdot \vec{n}(f(iy)) \\
&\qquad = e^{ay} \sqrt{a^2+1} \mu e^{(a+i)y} \cdot \vec{n}(f(iy)) = \mu e^{2ay} \quad \text{for} \  y\in\R.
\end{aligned}
\end{equation}
Let us observe that the tangential jump condition \eqref{skoki-tan} together with \eqref{eq-bound-1b} and \eqref{eq-bound-2b}, yields
\begin{equation}
\begin{aligned}\label{eq-39aa}
& \tilde{w}_{2}\left(-\frac{2a\pi}{1+a^2} + i\left(y + \frac{2\pi}{1+a^2}\right)\right) - \tilde{w}_{2}(iy) \\
& \qquad = e^{ay}\sqrt{1+a^2}(w^{R}(f(iy))-w^{L}(f(iy)))\cdot\vec{\tau}(f(iy)) \\
& \qquad = e^{ay}\sqrt{1+a^2}\gamma(f(iy)) \\
& \qquad = 2ag e^{2ay},
\end{aligned}
\end{equation}
where the last equality is a consequence of the formula \eqref{gestosc}. Therefore, from \eqref{eq-36aa}, \eqref{eq-36bb} and \eqref{eq-39aa} it follows that the function $\tilde w$ satisfies the following boundary conditions
\begin{align}\label{eq-36}
& \tilde{w}_{1}(iy) = \mu e^{2ay}, \quad \tilde{w}_{1}\left(-\frac{2a\pi}{1+a^2} + iy\right) = \mu e^{2a(y - \frac{2\pi}{1+a^2})}, \\ \label{eq-39}
& \tilde{w}_{2}\left(-\frac{2a\pi}{1+a^2} + i\left(y + \frac{2\pi}{1+a^2}\right)\right) - \tilde{w}_{2}(iy) = 2ag e^{2ay}
\end{align}
for $y\in\R$. To derive the transformed decay condition, we combine \eqref{decay} with \eqref{lim-f1} and the velocity transformations \eqref{vpierwsza}, \eqref{vdruga} obtaining the limit
\begin{align}\label{v-decay}
\tilde{w}(x+iy)e^{-ay} \to 0, \quad y\to-\infty,
\end{align}
which holds uniformly for  $-2\pi a/(1+a^2)<x<0$. 

Let us observe that \eqref{eq-39} specifies only the difference in $\tilde{w}_2$ between two points, rather than prescribing its absolute value along the boundary. Consequently, the resulting identification problem does not fall into the standard boundary‐value framework.

\subsection{Proof of the uniqueness part of Theorem \ref{glowne}}
Using the conformal transformation $f$, we have just reduced the proof of the uniqueness part of Theorem \ref{glowne} to the following proposition.
\begin{prop}\label{tresc}
Let $\tilde{w}^j = \tilde{w}^{j}_{1}+i\tilde{w}^{j}_{2}$, for $j = 1, 2$, be two antiholomorphic velocities in the strip $\SS$ satisfying the boundary conditions \eqref{eq-36} and the non-standard condition \eqref{eq-39}. Furthermore, assume that for any $j=1,2$ and $-2\pi a/(1+a^2)<x<0$ we have the pointwise limit
\begin{equation}\label{obietnica}
\tilde{w}^j(x+iy)\rightarrow 0\quad \text{as} \ \ y\to -\infty.
\end{equation}
Then $\tilde{w}^1 = \tilde{w}^2$ on $\SS$.
\end{prop}
\begin{proof}
Let us consider the mappings 
\begin{align*}
\PP_{\pm}(z) = -z^{*} - \frac{2a\pi}{1+a^2} \pm \frac{2\pi i}{1+a^2}, \quad z\in\SS, 
\end{align*}
which are well-defined transformations of the strip $\SS$ onto itself. More precisely, each $\PP_{\pm}$ is constructed as the composition of a reflection with respect to the symmetry axis of the strip $\{\mathrm{Re}\,z = - \tfrac{a\pi}{1+a^2}\}$ followed by a vertical shift by the vector $\pm\tfrac{2\pi i}{1+a^2}$. Moreover, the following equality holds 
\begin{equation}\label{eq-p1}
\PP_{+}(\PP_{-}(z)) = z, \quad z\in\SS
\end{equation}
and 
\begin{equation}\label{eq-p2}
\PP_{-}^{n}(z) = z - \frac{2\pi n i}{1+a^2} \quad \text{ for } z\in\SS\text{ and even } n\ge 1.
\end{equation}
Let us define the auxiliary functions $$g^j(z) := \tilde{w}_{1}^j(\PP_{+}(z)) + i\tilde{w}_{2}^j(\PP_{+}(z)), \quad z\in\SS, \ j=1,2.$$ 
Since the function $(w^{j})^{*}$ is holomorphic, the maps $g^j$ also have this property, as they satisfy the following Cauchy-Riemann equations 
\begin{equation*}
(\mathrm{Re}\, g^{j})_{x}(z) = -\tilde{w}_{1,x}^j(\PP_{+}(z)) = \tilde{w}_{2,y}^j(\PP_{+}(z)) = (\mathrm{Im}\, g^{j})_{y}(z)
\end{equation*}
and
\begin{equation*}
(\mathrm{Re}\, g^{j})_{y}(z) = \tilde{w}_{1,y}^j(\PP_{+}(z)) = \tilde{w}_{2,x}^j(\PP_{+}(z)) = -(\mathrm{Im}\, g^{j})_{x}(z).
\end{equation*}
Then we can define the holomorphic function on the strip $\SS$ as follows
\begin{equation*}
\begin{aligned}
h^j(z) & := (w^j(z))^{*} + g^j(z) \\
& = \tilde{w}_{1}^j(z) + \tilde{w}_{1}^j(\PP_{+}(z)) - i \left(\tilde{w}_{2}^j(z) - \tilde{w}_{2}^j(\PP_{+}(z))\right).
\end{aligned}
\end{equation*}
By \eqref{eq-36} and \eqref{eq-39}, the following boundary conditions are satisfied 
\begin{equation}
\begin{aligned}\label{eq-50}
h^j(iy) &= \tilde{w}_{1}^j(iy) + \tilde{w}_{1}^j\left(-\frac{2a\pi}{1+a^2}+ i\left(y + \frac{2\pi}{1+a^2}\right)\right) \\
&\quad - i \left(\tilde{w}_{2}^j(iy) - \tilde{w}_{2}^j\left(-\frac{2a\pi}{1+a^2}+ iy + \frac{2\pi i}{1+a^2}\right)\right) \\
&= 2\mu e^{2ay} +i 2a g e^{2ay} \quad \text{for} \ \ y\in\R.
\end{aligned}
\end{equation}
Let us now define the function
\begin{equation*}
h(z) := h^1(z) - h^2(z), \quad z\in\SS,
\end{equation*}
which is holomorphic on $\SS$ and continuous up to the boundary $\partial\SS$. Moreover, by \eqref{eq-50}, it is known that $h(iy) = 0$ for $y\in\R$. By applying the Schwarz reflection principle (see Proposition \ref{holojedno}), we conclude that $h^1(z) - h^2(z) = h(z)=0$ for $z\in\SS$. 

To prove the uniqueness of the velocity field, let us note that
\begin{equation}\label{formula-1}
\tilde{w}_{2}^j(\PP_{+}(z)) - \tilde{w}_{2}^j(z) = \mathrm{Im}\,h^{j}(z), \quad j=1,2, \ z\in\SS. 
\end{equation}
Applying the mappings $\PP_{-}, \PP_{-}^{2}, \ldots, \PP_{-}^{n}$ to \eqref{formula-1}, where $n\ge 1$ is an even number, and using \eqref{eq-p1}, we obtain the following sequence of equations
\begin{equation}
\begin{aligned}\label{formula-2}
\tilde{w}_{2}^j(z) - \tilde{w}_{2}^j(\PP_{-}(z)) & = \mathrm{Im}\, h^j(\PP_{-}(z)), \\
\tilde{w}_{2}^j(\PP_{-}(z)) - \tilde{w}_{2}^j(\PP_{-}^{2}(z)) & = \mathrm{Im}\, h^j(\PP_{-}^{2}(z)), \\
& \vdots \\
\tilde{w}_{2}^j(\PP_{-}^{n-1}(z)) - \tilde{w}_{2}^j(\PP_{-}^{n}(z)) & = \mathrm{Im}\, h^j(\PP_{-}^{n}(z)).
\end{aligned}
\end{equation}
Adding the formulas \eqref{formula-2} side-by-side and using \eqref{eq-p2}, we obtain
\begin{equation*}
\begin{aligned}
\tilde{w}_{2}^j(z) & = \tilde{w}_{2}^j(\PP_{-}^{n}(z)) + \sum_{k=1}^{n}\mathrm{Im}\, h^j(\PP_{-}^{k}(z)) \\
& = \tilde{w}_{2}^j\left(z - \frac{2\pi n i}{1+a^2}\right) + \sum_{k=1}^{n}\mathrm{Im}\, h^j(\PP_{-}^{k}(z))
\end{aligned}
\end{equation*}
for $j=1,2$ and even $n\ge 1$. Passing in the above equation to the limit with $n\to \infty$ and using the assumption \eqref{obietnica}, we conclude that 
\begin{align*}
\tilde{w}_{2}^j(z) = \sum_{k=0}^{\infty}\mathrm{Im}\, h^j(\PP_{-}^{k}(z)), \quad j=1,2, \ z\in\SS.
\end{align*}
Recalling that $h^1(z) = h^2(z)$ for $z\in\SS$, we obtain
\begin{align*}
\tilde{w}_{2}^1(z) & = \sum_{k=0}^{\infty}\mathrm{Im}\, h^1(\PP_{-}^{k}(z)) = \sum_{k=0}^{\infty}\mathrm{Im}\, h^2(\PP_{-}^{k}(z)) = \tilde{w}_{2}^2(z), \quad z\in\SS.
\end{align*}
It remains to show that the real parts of the velocity fields $\tilde w^{j}$ coincide. Since the function $(\tilde{w}^{1}-\tilde{w}^{2})^{*}$ is holomorphic on the strip $\SS$ and its imaginary part vanishes, from the Cauchy-Riemann equation we obtain $$(\tilde{w}^{1}_{1})_{y}(z) = (\tilde{w}^{2}_{1})_{y}(z), \quad z\in\SS,$$ which together with \eqref{obietnica} gives
$\tilde{w}^{1}_{1} = \tilde{w}^{2}_{1}$ on $\SS$ as desired. Thus the proof of the proposition is completed. 
\end{proof}
The following result will be useful in Section \ref{wiele_spiral}.
\begin{cor}\label{cor-eq-zero}
Let $\tilde{w} = \tilde{w}_{1}+i\tilde{w}_{2}$, be an antiholomorphic velocity field in the strip $\SS$ satisfying the boundary conditions 
\begin{gather}
\tilde{w}_{1}(iy) = \tilde{w}_{1}\left(-\frac{2a\pi}{1+a^2} + iy\right) = 0, \\ 
\tilde{w}_{2}(iy) = \tilde{w}_{2}\left(-\frac{2a\pi}{1+a^2} + i\left(y + \frac{2\pi}{1+a^2}\right)\right)
\end{gather}
for $y\in\R$. Furthermore, assume that for any $-2\pi a/(1+a^2)<x<0$ we have the pointwise limit
\begin{equation*}
\tilde{w}(x+iy)\rightarrow 0\quad \text{as} \ \ y\to -\infty.
\end{equation*}
Then $\tilde{w} = 0$ on $\SS$.
\end{cor}
\begin{proof}
Let us observe that the functions $\tilde w^{1}:= \tilde w$ and $\tilde w^{2}:=0$ satisfy the conditions \eqref{eq-36} and \eqref{eq-39} with the parameters $ \mu =  g = 0$. Then Proposition \ref{tresc} implies that $\tilde{w}^1 = \tilde w^{2} = 0$ on $\SS$ as desired.
\end{proof}

\section{Determination of the velocity in the strip}\label{sec-4}
In the next section, we will present an alternative derivation of the velocity profile \eqref{form-w} for the Prandtl spiral, complementing the approach in \cite{CKO}. Our procedure is as follows. We use the conformal map from Lemma \ref{mapping} to transfer the problem to the strip 
$\SS$. Then, we compute the velocity field on $\SS$ and we return to the original spiral coordinates. We shall verify that this second method reproduces the formula \eqref{form-w} obtained in \cite{CKO} by the contour integration method. The present section is devoted to determining the velocity in the strip. Moreover, as a by-product, we also provide an explicit formula for the function $h^{j}$ used in the proof of Theorem \ref{glowne}.

In this section, we additionally assume that the parameters of the spiral satisfy the condition (see \cite[eq. (1.18)]{CKO} or \cite{Elling})
\begin{equation}\label{eq-40}
a^2 + 1 - 2\mu + 2a\mu i = -2a^2 g \coth(\pi A). 
\end{equation}
It is known (see \cite{CKO}) that there exist values $g\in\R\setminus\{0\}$ and $\mu\in\R$ for which \eqref{eq-40} holds, provided $a>0$ is sufficiently large. In this case, Prandtl's spiral is a solution to the 2D Euler equations. 
\begin{prop}\label{wpasie}
Let $w$ be a two-dimensional weakly divergence-free velocity field satisfying \eqref{decay}, such that $\curl w=\gamma\delta_{\Sigma}$ in the sense of distributions. If the condition \eqref{eq-40} is satisfied, then the corresponding velocity field $\tilde{w}=\tilde{w}_1+i\tilde{w}_2$ in the strip $\SS$ is uniquely determined by the formulas
\begin{equation}\label{w-explicit}
\begin{aligned}
\tilde{w}_1(x,y)& =\mu (C_{a}\sin(2ax) + \cos(2ax) )e^{2ay}, \\
\tilde{w}_2(x,y)& =\mu (\sin(2ax) - C_{a} \cos(2ax) ) e^{2ay}
\end{aligned}
\end{equation}
for $(x,y)\in\SS$, where we define
\begin{align*}
C_{a}:=\frac{\cos\left(\frac{4a^2\pi}{1+a^2}\right) - e^{-\frac{4a\pi}{1+a^2}}}{\sin\left(\frac{4a^2\pi}{1+a^2}\right)}.
\end{align*}
\end{prop}
\begin{proof}
We begin by showing how to derive the formulas \eqref{w-explicit}. Let us recall from \cite{CKO} that the real part of \eqref{eq-40} is precisely the velocity matching condition \eqref{vel_match}. Hence, by results of Section \ref{sec-frame}, the transformed velocity field $\tilde w$ given by \eqref{transform} satisfies on the strip $\SS$ the boundary conditions \eqref{eq-36}, \eqref{eq-39} and the decay \eqref{v-decay}. We postulate that $\tilde{w}_{1}(x,y) = k(x)e^{2ay}$ for $(x,y)\in\SS$. Since $\tilde{w}_{1}$ is a harmonic function, we obtain
\begin{equation}\label{eq-72}
k''(x) + 4a^2k(x) = 0, \quad x\in\left(-\frac{2\pi a}{1+a^2}, 0\right).
\end{equation}
Moreover, using \eqref{eq-36}, we determine the boundary conditions
\begin{equation}\label{eq-73}
\quad k(0) = \mu \quad\text{and}\quad k\left(-\frac{2a\pi}{1+a^2}\right) = \mu e^{-\frac{4a\pi}{1+a^2}}.
\end{equation}
The equation \eqref{eq-72} admits the general solution
\begin{equation*}
k(x) = A_{1} \sin(2ax) + A_{2} \cos(2ax),
\end{equation*}
which, after applying the boundary conditions \eqref{eq-73}, gives $A_{2}=\mu$, and
\begin{equation*}
A_{1} \sin\left(\frac{4a^2\pi}{1+a^2}\right) = \mu \left( \cos\left(\frac{4a^2\pi}{1+a^2}\right) - e^{-\frac{4a\pi}{1+a^2}} \right).
\end{equation*}
Substituting the obtained coefficients back into the expression for $k(x)$ yields
\begin{equation*}
\tilde{w}_{1}(x,y) = \mu \left( C_{a}  \sin(2ax) + \cos(2ax) \right)e^{2ay}, \quad (x,y)\in\SS.
\end{equation*}
Since the function $\tilde w^{*} = \tilde{w}_{1} - i\tilde{w}_{2}$ is holomorphic on the strip, from the Cauchy-Riemann equations we obtain
\begin{equation*}
(\tilde{w}_{2})_y(x,y) = -(\tilde{w}_{1})_x(x,y) = 2a\mu\left(\sin(2ax)-C_{a}\cos(2ax)\right)e^{2ay},
\end{equation*}
which together with the decay \eqref{obietnica} gives
\begin{equation*}
\tilde{w}_{2}(x,y) = \mu \left( \sin(2ax) - C_{a} \cos(2ax) \right) e^{2ay}, \quad (x,y)\in\SS.
\end{equation*}
By Proposition \ref{tresc}, it remains to check that the velocity field $\tilde w$ given by the formulas \eqref{w-explicit} satisfies the boundary conditions \eqref{eq-36}, \eqref{eq-39} and the decay \eqref{v-decay}. Let us observe that, by the construction of the function, $w$ the conditions \eqref{eq-36} and \eqref{v-decay} are satisfied. Moreover, we have
\[
\tilde{w}_{2}(iy) = -\mu C_{a}e^{2ay}, \quad y\in\R
\]
and 
\begin{equation*}
\begin{aligned}
& \tilde{w}_{2}\left(-\frac{2a\pi}{1+a^2} + iy\right) = -\mu \left( C_{a}\cos\left(-\frac{4a^2\pi}{1+a^2}\right) - \sin\left(-\frac{4a^2\pi}{1+a^2}\right) \right)e^{2ay} \\
& \qquad = -\mu \frac{1 - e^{-\frac{4a\pi}{1+a^2}} \cos\left(\frac{4a^2\pi}{1+a^2}\right)}{\sin\left(\frac{4a^2\pi}{1+a^2}\right)} e^{2ay}, \quad y\in\R.
\end{aligned}
\end{equation*}
Subtracting the above equations gives
\begin{equation}
\begin{aligned}\label{eq-roz}
& \tilde{w}_{2}\left(-\frac{2a\pi}{1+a^2} + i\left(y + \frac{2\pi}{1+a^2}\right)\right) - w_{2}(iy)  \\
& \qquad = -\mu \frac{e^{\frac{4a\pi}{1+a^2}} -2\cos\left(\frac{4a^2\pi}{1+a^2}\right) + e^{-\frac{4a\pi}{1+a^2}}}{\sin\left(\frac{4a^2\pi}{1+a^2}\right)}e^{2ay},
\end{aligned}
\end{equation}
On the other hand, it is not difficult to check that the imaginary part of the equation \eqref{eq-40} takes the form
\begin{equation*}
2ag = \mu\frac{2\cos\left(\frac{4\pi a^2}{1+a^2}\right) - e^{\frac{-4\pi a}{1+a^2}} - e^{\frac{4\pi a}{1+a^2}}}{\sin\left(\frac{4\pi a^2}{1+a^2}\right)},
\end{equation*}
which, together with \eqref{eq-roz}, shows that $\tilde w_{2}$ satisfies the condition \eqref{eq-39}, as required. Thus the proof of the proposition is completed. 
\end{proof}
\begin{rem}
As we see in the proof of Proposition \ref{wpasie}, $\tilde{w}_1$ is a holomorphic function in the strip $\SS$. One could be tempted to obtain its uniqueness (and hence the uniqueness of $\tilde{w}$) by applying the results of Widder \cite{Widder} concerning the harmonic functions in the strip. We encourage the reader to check that, in this case, uniqueness holds in the class of functions satisfying the decay of the form $O(e^{(a+1/a)|y|})$ as $|y|\rightarrow \infty$. This means that the uniqueness result could be obtained only for $a\leq 1$, which is an essentially weaker result than the one we present. \hfill $\square$
\end{rem}
\begin{rem}
It can be verified that under the condition \eqref{eq-40}, the functions $h^{1}=h^{2}$ from the proof of Theorem \ref{glowne} can be expressed explicitly as
\begin{equation*}
\begin{aligned}
h^{1}(z) &= \mu \!\left( \frac{2 \cos\left(\frac{4a^2\pi}{1+a^2}\right) \!-\!e^{-\frac{4a\pi}{1+a^2}} \!-\! e^{\frac{4a\pi}{1+a^2}}}{\sin\left(\frac{4a^2\pi}{1+a^2}\right)} \sin(2ax) \!+\! 2 \cos(2ax)\right) \!e^{2ay} \\
& \quad +i \mu \!\left( \frac{2 \cos\left(\frac{4a^2\pi}{1+a^2}\right) \!-\! e^{-\frac{4a\pi}{1+a^2}} \!-\! e^{\frac{4a\pi}{1+a^2}}}{\sin\left(\frac{4a^2\pi}{1+a^2}\right)}  \cos(2ax) \!-\! 2 \sin(2ax)\right) \!e^{2ay}
\end{aligned}
\end{equation*}
for $z = x+iy\in\SS$. \hfill $\square$
\end{rem}

\section{Alternative way of finding the formula for the velocity}\label{sec-velocity}
This section is devoted to the proof of the last assertion of Theorem \ref{glowne}. Under the assumption \eqref{im_Pr}, we check that the expression for $\tilde{w}$ yields exactly the same formula for the profile $w$ as in \cite{CKO}. To this end, we apply the inverse of $f$ from Lemma \ref{mapping} to the complex potential $\tilde\Phi=\tilde\Phi_{1} + i\tilde\Phi_{2}$ of the vector field $\tilde{w}$ that we derived in Proposition \ref{wpasie}. In this case
\begin{align*}
\tilde\Phi_{1}(z) & = \mu\left(-C_{a}\frac{\cos(2ax)}{2a} + \frac{\sin(2ax)}{2a}\right)e^{2ay},
\end{align*}
and
\begin{align*}
\tilde\Phi_{2}(z) & = \mu\left(C_{a}\frac{\sin(2ax)}{2a} + \frac{\cos(2ax)}{2a}\right)e^{2ay}
\end{align*}
for $z=x+iy\in\SS$. Consequently, the complex potential is
\begin{equation}\label{eq-ph}
\begin{aligned}
\tilde \Phi(z) & = -\frac{\mu}{2a}\left(\frac{\cos\left(\frac{4a^{2}\pi}{1+a^{2}}\right) - e^{-\frac{4a\pi}{1+a^{2}}}}{\sin\left(\frac{4a^{2}\pi}{1+a^{2}}\right)} -
i\right)e^{-2axi}e^{2ay} \\
& = \frac{\mu}{2a}\left(\frac{e^{-\frac{4a\pi}{1+a^{2}}} - e^{-\frac{4a^{2}\pi}{1+a^{2}} i}}{\sin\left(\frac{4a^{2}\pi}{1+a^{2}}\right)}\right)e^{-2aiz}.
\end{aligned}
\end{equation}
Using the explicit formula for $f^{-1}$ from Lemma \ref{mapping}, we obtain
\begin{equation}\label{eq-inv-1}
\begin{aligned}
-2aif^{-1}(z) & = iA(\ln r + i(\theta - 2\pi(J(r,\theta) - 1))) \\
& = iA(\ln r + i(\theta - 2\pi J(r,\theta))) - 2\pi A,
\end{aligned}
\end{equation}
where we recall that $A = -2ai/(a+i)$. Combining \eqref{eq-ph} and \eqref{eq-inv-1}, gives
\begin{equation}\label{eq-ph-inv}
\begin{aligned}
\tilde\Phi(f^{-1}(z)) & = \frac{\mu}{2a}\left(\frac{e^{-\frac{4a\pi}{1+a^{2}}} - e^{-\frac{4a^{2}\pi}{1+a^{2}} i}}{\sin\left(\frac{4a^{2}\pi}{1+a^{2}}\right)}\right)e^{-2ai f^{-1}(z)} \\
& = \frac{\mu}{2a}\left(\frac{e^{-\frac{4a\pi}{1+a^{2}}} - e^{-\frac{4a^{2}\pi}{1+a^{2}} i}}{\sin\left(\frac{4a^{2}\pi}{1+a^{2}}\right)}\right) e^{-2\pi A}
e^{iA(\ln r + i(\theta - 2\pi J(r,\theta)))}
\end{aligned}
\end{equation}
for $z=re^{i\theta}\in \RR$. Then the following calculation shows
\begin{align*}
& \frac{e^{2\pi A}}{1-e^{2\pi A}} = \frac{e^{\pi A}}{e^{-\pi A}-e^{\pi A}} = -\frac{e^{\frac{-2\pi a}{1+a^{2}}(1+ai)}}{e^{\frac{-2\pi a}{1+a^{2}}(1+ai)} - e^{\frac{2\pi a}{1+a^{2}}(1+ai)}} \\
& \quad = -\frac{e^{\frac{-4\pi a}{1+a^{2}}} - e^{\frac{-4\pi a^{2}}{1+a^{2}}i}}{e^{\frac{-4\pi a}{1+a^{2}}} + e^{\frac{4\pi a}{1+a^{2}}} - 2\cos\left(\frac{4\pi a^{2}}{1+a^{2}}\right)}
 = \frac{e^{\frac{-4\pi a}{1+a^{2}}} - e^{\frac{-4\pi a^{2}}{1+a^{2}}i}}{2\cos\left(\frac{4\pi a^{2}}{1+a^{2}}\right) - e^{\frac{-4\pi a}{1+a^{2}}} - e^{\frac{4\pi a}{1+a^{2}}}}.
\end{align*}
Substituting the above equality in \eqref{im_Pr}, we infer that
\begin{align*}
\frac{e^{2\pi A}}{1-e^{2\pi A}} & = \frac{e^{\frac{-4\pi a}{1+a^{2}}} - e^{\frac{-4\pi a^{2}}{1+a^{2}}i}}{2\cos\left(\frac{4\pi a^{2}}{1+a^{2}}\right) - e^{\frac{-4\pi a}{1+a^{2}}} - e^{\frac{4\pi a}{1+a^{2}}}}
 = \frac{\mu}{2ag}\frac{e^{\frac{-4\pi a}{1+a^{2}}} - e^{\frac{-4\pi a^{2}}{1+a^{2}}i}}{\sin\left(\frac{4\pi a^2}{1+a^2}\right)}
\end{align*}
and consequently
\begin{align}\label{eq-11aa}
\frac{\mu}{2a}\left(\frac{e^{-\frac{4a\pi}{1+a^{2}}} - e^{-\frac{4a^{2}\pi}{1+a^{2}} i}}{\sin\left(\frac{4a^{2}\pi}{1+a^{2}}\right)}\right) e^{-2\pi A}
=\frac{g}{1-e^{2\pi A}}.
\end{align}
Combining \eqref{eq-11aa} and \eqref{eq-ph-inv} yields
\begin{align*}
\tilde \Phi(f^{-1}(z)) = \frac{g}{1-e^{2\pi A}}e^{iA(\ln r + i(\theta - 2\pi J(r,\theta)))}, \quad z=re^{i\theta}\in\RR,
\end{align*}
which coincides exactly with the potential given in \cite[eq. (1.31)]{CKO} for $M=1$ and $\theta_0=0$. It therefore reproduces the potential and velocity field $w$ as in \eqref{eq-phi} and \eqref{form-w}, respectively, completing the proof of Theorem \ref{glowne}.  \hfill $\square$

\section{Uniqueness for the family of logarithmic spirals}\label{wiele_spiral}
In this section, we consider flows whose vorticity is supported on a family of $M\ge 2$ concentric logarithmic spirals, described by the following equations
\begin{equation*}
Z_m(\theta , t) = t^\mu e^{a (\theta - \theta_m )} e^{i\theta }, \quad 
\Gamma_m (\theta ,t) = g_m t^{2\mu -1}  e^{2a (\theta - \theta_m )}, \quad \theta \in \R, \ t>0,
\end{equation*}
where $Z_m(\theta , t)$ represents the parametrization of the $m$-th spiral with respect to the angular parameter and $\Gamma_m (\theta , t)$ denotes the circulation function. In the above formulas, the quantities involved are such that
\begin{equation}\label{log_spirals_parameters}
a>0, \quad \mu \in \R , \quad g_m \in \R\setminus\{0\}, \quad \theta_m \in [0,2\pi ),
\end{equation}
where $0\le m \le M-1$. Without loss of generality, we assume that
\[
0 = \theta_0 < \theta_1 < \ldots < \theta_{M-1} < 2\pi. 
\]
Let $\Sigma_{m}(t)$ be the spiral parameterized by the map $Z_{m}(\theta,t)$ and define $\Sigma_{F}(t) := \Sigma_{0}(t)\cup\ldots\cup\Sigma_{M-1}(t)$. Throughout the remainder of this paper, we focus on the self-similar velocity field given by 
\begin{equation*}
v(z,t)=t^\mu w(z/t^\mu), \quad z\in \RR_{F}(t) := \C\setminus(\Sigma_{F}(t)\cup\{0\})
\end{equation*}
such that the profile vector field satisfies, on $\Sigma_{m} := \Sigma_{m}(1)$ for $0\le m \le M-1$, the velocity matching condition 
\begin{equation}\label{vel_match-m}
\vec{n}(z)\cdot\left(w(z)-\mu z\right)=0, \quad z\in \Sigma_{m},
\end{equation}
the continuity of normal components  
\begin{equation}\label{skoki-nor-m}
(w^R(z)-w^L(z))\cdot \vec{n}(z) = 0, \quad  z\in \Sigma_{m}, 
\end{equation}
as well as the tangential jump conditions
\begin{equation}\label{skoki-tan-m}
(w^R(z)-w^L(z))\cdot\vec{\tau}(z) = \gamma_{m}(z), \quad z\in \Sigma_{m}, 
\end{equation}
where the density function $\gamma_{m}:\Sigma_{m}\to\R$ is given by the formula 
\begin{equation}\label{gestosc2bb-m}
\gamma_{m}(Z_{m}(\theta)):=\frac{\partial_\theta \Gamma_{m}(\theta)}{|\partial_\theta Z_{m}(\theta)|} =
\frac{2ag_{m}}{\sqrt{1+a^{2}}} e^{a(\theta-\theta_{m})}, \quad \theta\in\R.
\end{equation}
Moreover, we assume that $w$ satisfies the following decay
\begin{equation}\label{decay-m}
|w(z)| \lesssim |z|, \quad z\to 0, \ z\in\RR_{F}:=\RR_{F}(1).
\end{equation}
Observe that, by expressing the unit tangent and normal vectors on the spiral $\Sigma_{m}$, we obtain
\begin{align}\label{vec-tan-norm}
\vec{\tau}(Z_{m}(\theta)) = \frac{a+i}{\sqrt{a^2+1}} e^{i\theta} \quad\text{and}\quad \vec{n}(Z_{m}(\theta))=\frac{1-ia}{\sqrt{a^2+1}} e^{i\theta} \quad\text{for} \ \ \theta\in\R.
\end{align}
Let us define the lines
\begin{align*}
\ell_{m} := \left\{\mathrm{Re}\, z = -\frac{a\theta_{m}}{1+a^{2}}\right\}, \quad 0\le m \le M,
\end{align*}
where we set $\theta_{M}:=2\pi$. It is not difficult to check that 
\begin{align}\label{eq-f-spiral}
f\left(-\frac{a\theta_{m}}{1+a^{2}} + iy\right) 
= e^{a(y+\frac{a^{2}\theta_{m}}{1+a^{2}} - \theta_{m})} e^{i(y+\frac{a^{2}\theta_{m}}{1+a^{2}})}, \quad y\in \R,
\end{align}
which gives $f(\ell_{m})=\Sigma_{m}$ for $0\le m\le M-1$. Hence, $f$ is a biholomorphic map between $\SS_{F}:=\SS\setminus(\ell_{1}\cup\ldots\cup\ell_{M-1})$ and $\RR_{F}$. If we consider the function 
\begin{align*}
f_{m}(z) := f\left(-\frac{a\theta_{m}}{1+a^{2}} + z\right), \quad z\in i\R,
\end{align*}
then the equation \eqref{eq-f-spiral}, gives
\begin{align*}
f_{m}(iy) = Z_{m}\left(y+\frac{a^{2}\theta_{m}}{1+a^{2}}\right), \quad y\in\R
\end{align*}
and consequently, by \eqref{gestosc2bb-m}, we have
\begin{align}\label{cond-gamma-m}
\gamma(f_{m}(iy)) = \frac{2ag_{m}}{\sqrt{1+a^{2}}} e^{a(y + \frac{a^{2}\theta_{m}}{1+a^{2}} - \theta_{m})} = 
\frac{2ag_{m}}{\sqrt{1+a^{2}}} e^{a(y - \frac{\theta_{m}}{1+a^{2}})}, \quad y\in\R.
\end{align}
On the other hand, from \eqref{vec-tan-norm} it follows that
\begin{align*}
n(f_{m}(iy)) = \frac{1-ia}{\sqrt{a^2+1}} e^{i(y+\frac{a^{2}\theta_{m}}{1+a^{2}})}, \quad y\in\R,
\end{align*}
which together with the condition \eqref{vel_match-m} and the equation \eqref{vpierwsza}, gives
\begin{equation}
\begin{aligned}\label{eq-bound-11}
\tilde{w}_{1}^{L}\left(-\frac{a\theta_{m}}{1+a^{2}} + i y\right) & = e^{-\frac{a\theta_{m}}{1+a^{2}} + a y}\sqrt{a^2+1}w^{L}(f_{m}(iy))\cdot \vec{n}(f_{m}(iy)) \\
& = \mu e^{-\frac{a\theta_{m}}{1+a^{2}} + a y} \sqrt{a^2+1} f_{m}(iy) \cdot \vec{n}(f_{m}(iy)) \\
& = \mu e^{-\frac{a\theta_{m}}{1+a^{2}} + a y} \sqrt{a^2+1} \,\mathrm{Re}\!\left(e^{-\frac{a\theta_{m}}{1+a^{2}}+ay} \frac{1-ia}{\sqrt{a^2+1}}\right) \\
& = \mu e^{2a(y-\frac{\theta_{m}}{1+a^{2}})}, \quad y\in\R.
\end{aligned}
\end{equation}
It is not difficult to check that the analogous calculations provide
\begin{align}\label{eq-bound-22}
\tilde{w}_{1}^{R}\left(-\frac{a\theta_{m}}{1+a^{2}} + i y\right) = \mu e^{2a(y-\frac{\theta_{m}}{1+a^{2}})}, \ \ y\in\R.
\end{align}
To derive the condition for the imaginary part $\tilde w_{2}$ we use \eqref{vdruga} to obtain 
\begin{align*}
\tilde{w}_{2}^{L}\left(-\frac{a\theta_{m}}{1+a^{2}} + i y\right) & = e^{-\frac{a\theta_{m}}{1+a^{2}} + a y} \sqrt{a^2+1}  w^{L}(f_{m}(iy)) \cdot \vec{\tau}(f_{m}(iy))
\end{align*}
and
\begin{align*}
\tilde{w}_{2}^{R}\left(-\frac{a\theta_{m}}{1+a^{2}} + i y\right) & = e^{-\frac{a\theta_{m}}{1+a^{2}} + a y} \sqrt{a^2+1}  w^{R}(f_{m}(iy)) \cdot \vec{\tau}(f_{m}(iy)) 
\end{align*}
for $y\in\R$. Hence, by the condition \eqref{skoki-tan-m} and the equation \eqref{cond-gamma-m}, we have 
\begin{equation}
\begin{aligned}\label{eq-bound-2bb}
&\tilde{w}_{2}^{R}\left(-\frac{a\theta_{m}}{1+a^{2}} + i y\right) - \tilde{w}_{2}^{L}\left(-\frac{a\theta_{m}}{1+a^{2}} + i y\right) \\
&\qquad = e^{-\frac{a\theta_{m}}{1+a^{2}} + a y} \sqrt{a^2+1}  (w^{R}(f_{m}(iy)) - w^{L}(f_{m}(iy))) \cdot \vec{\tau}(f_{m}(iy)) \\
&\qquad = e^{-\frac{a\theta_{m}}{1+a^{2}} + a y} \sqrt{a^2+1}\gamma_{m}(f_{m}(iy)) \\
&\qquad = e^{-\frac{a\theta_{m}}{1+a^{2}} + a y} \sqrt{a^2+1} \frac{2ag_{m}}{\sqrt{1+a^{2}}} e^{a(y - \frac{\theta_{m}}{1+a^{2}})} \\
&\qquad = 2ag_{m} e^{2a(y - \frac{\theta_{m}}{1+a^{2}})}, \quad y\in\R.
\end{aligned}
\end{equation}
From \eqref{eq-bound-11}, \eqref{eq-bound-22} and \eqref{eq-bound-2bb}, it follows that the transformed velocity field $\tilde w$ is an antiholomorphic function defined on the set $\SS_{F}$ and satisfies on the lines $\ell_{m}$ the following jump conditions 
\begin{align}\label{eq-bound}
& \tilde{w}_{1}^{R}(z) = \tilde{w}_{1}^{L}(z) = \mu e^{2a(y-\frac{\theta_{m}}{1+a^{2}})}, \\ \label{jump}
& \tilde{w}_{2}^{R}(z) - \tilde{w}_{2}^{L}(z) = 2ag_{m} e^{2a(y - \frac{\theta_{m}}{1+a^{2}})}, 
\end{align}
for $z=x+iy\in\ell_{m}$ and $1\le m \le M-1$. Using the same calculations as in Section \ref{sec-frame}, based on the jump relations \eqref{skoki-nor-m}, \eqref{skoki-tan-m}, and the formula \eqref{cond-gamma-m}, we conclude that the transformed velocity field $\tilde w$ satisfies the conditions 
\begin{align}\label{eq-36dd}
\tilde{w}_{1}(z) = \mu e^{2ay}, \ \ z\in\ell_{0}, \qquad 
\tilde{w}_{1}(z) = \mu e^{2a(y - \frac{2\pi}{1+a^2})}, \ \ z\in\ell_{M}
\end{align}
and
\begin{align}\label{eq-39dd}
\tilde{w}_{2}\left(-\frac{2a\pi}{1+a^2} + i\left(y + \frac{2\pi}{1+a^2}\right)\right) - \tilde{w}_{2}(iy) = 2ag_{0} e^{2ay}, \quad y\in i\R.
\end{align}
Moreover, analogous transformation of the decay \eqref{decay-m} leads to the limit 
\begin{align}\label{v-decay-2}
\tilde{w}(x+iy)e^{-ay} \to 0, \quad y\to-\infty,
\end{align}
which is uniform for $x\in\SS_{F}\cap \R$.
\begin{prop}
Let $\tilde{w}^j = \tilde{w}^{j}_{1}+i\tilde{w}^{j}_{2}$, for $j = 1, 2$, be two antiholomorphic velocities in the set $\SS_{F}$ satisfying the jump conditions \eqref{eq-bound}, \eqref{jump} on the lines $\ell_{m}$ for $1\le m \le M-1$, as well as the boundary conditions \eqref{eq-36dd}, \eqref{eq-39dd}. Furthermore, assume that for each $j=1,2$ we have the limit
\begin{equation*}
\tilde{w}^j(x+iy)\rightarrow 0\quad \text{as} \ \ y\to -\infty,
\end{equation*}
which is uniform for $x\in\SS_{F}\cap \R$. Then $\tilde{w}^1 = \tilde{w}^2$ on $\SS_{F}$.
\end{prop}
\begin{proof}
Let us define the holomorphic function on the set $\SS_{F}$ as follows
\begin{equation*}
h(z) := (\tilde w^1(z) - \tilde w^2(z))^{*}.
\end{equation*}
By the conditions \eqref{eq-bound} and \eqref{jump}, for any $1\le m \le M-1$, we have
\begin{align*}
h^{R}(z) - h^{L}(z) = ((\tilde w^1)^{R}(z) - (\tilde w^1)^{L}(z))^{*} - ((\tilde w^2)^{R}(z) - (\tilde w^2)^{L}(z))^{*} = 0
\end{align*}
for $z\in\ell_{m}$. Hence $h$ extends to a holomorphic function on the whole strip $\SS$ and satisfies the following limit
\begin{align*}
h(x+iy)\to 0, \quad y\to-\infty,
\end{align*}
for $-2\pi a/(1+a^{2}) <x <0$ such that $x\neq -a\theta_{m}/(1+a^{2})$ for $1\le m \le M-1$. Moreover, since the limit \eqref{v-decay-2} is uniform, it follows that
\begin{align*}
h^{R}\left(-\frac{a\theta_{m}}{1+a^{2}}+yi\right) = h^{L}\left(-\frac{a\theta_{m}}{1+a^{2}}+yi\right) \to 0, \quad y\to-\infty.
\end{align*}
Observe that, by the conditions \eqref{eq-bound}, the function $h = h_{1} + ih_{2}$ satisfies
\begin{align*}
h_{1}(iy) = h_{1}\left(-\frac{2a\pi}{1+a^2} + iy\right) = 0, \quad y\in\R,
\end{align*}
whereas the jump condition \eqref{jump} gives
\begin{align*}
h_{2}\left(-\frac{2a\pi}{1+a^2} + iy\right) - h_{2}(iy) = 0, \quad y\in\R.
\end{align*}
Applying Corollary \ref{cor-eq-zero} to the function $h^{*}$ gives $h^{*}(z) = 0$ for $z\in\SS$, and consequently $\tilde w^1 = \tilde w^2$ on the set $\SS_{F}$. Thus the proof is completed. 
\end{proof}
The considerations in Section \ref{wiele_spiral} are summarized in the following theorem.
\begin{theorem}\label{th-spiral-unique}
Let $w$ be a two-dimensional weakly divergence-free velocity field satisfying the velocity matching condition \eqref{vel_match-m} and the decay condition \eqref{decay-m}. Assume that 
\begin{equation}\label{curl-form}
\curl w = \sum_{k=0}^{M-1}\gamma_{k}\, \delta_{\Sigma_{k}}
\end{equation}
in the sense of distributions, where $\delta_{\Sigma_{k}}$ denotes the Dirac measure supported on the spiral $\Sigma_{k}$. Then the velocity field $w$ is uniquely determined. 
\end{theorem}

\section{Boundary problem for a family of logarithmic spirals}
In this section, we assume that $w$ is the velocity field satisfying, on each spiral $\Sigma_{m}$ and the continuity of normal components \eqref{skoki-nor-m}, the tangential jump conditions \eqref{skoki-tan-m}. We assume also that the decay condition \eqref{v-decay-2} holds and the spirals parameters are constrained by the following system of discrete equations (see \cite[eq. (1.18)]{CKO})
\begin{equation}\label{eq-disc2}
\frac{1}{\sinh (\pi A) }\sum_{k=0}^{M-1} \mathcal{A}_{mk}g_{k} = -(a^2+1-2\mu +2a\mu i )/2a^2
\end{equation}
for $0 \le m \le M-1$, where $A = -2ai/(a+i)$ and 
\begin{equation}\label{def_of_Amk} 
\mathcal{A}_{mk}:= e^{A(\theta_{k}-\theta_{m})}
\left\{\begin{aligned}
& e^{-\pi A}, &&  k>m ,\\
& \cosh(\pi A), && k=m , \\
& e^{\pi A},  && k<m
\end{aligned}\right.
\end{equation}
for $0 \le k,m\le M-1$. From \cite[Theorem 1.3]{CKO} it follows that the system formed by the real parts of the equations \eqref{def_of_Amk} is equivalent to the velocity matching condition \eqref{vel_match-m}. Hence, by results of Section \ref{wiele_spiral}, the transformed velocity field $\tilde w$ given by \eqref{transform} satisfies on the strip $\SS_{F}$ the boundary conditions \eqref{eq-bound}\,--\,\eqref{eq-39dd} and exhibits the decay property \eqref{v-decay-2}. We look for the formula for the velocity $\tilde w = \tilde{w}_{1} + i\tilde{w}_{2} $ postulating that the first component has the form  
\begin{align}\label{w1}
\tilde w_{1}(z) := \sum_{l=0}^{M-1} e^{r_{l}(z)}(A_{1,l}\sin\varphi_{l}(z) + A_{2,l}\cos\varphi_{l}(z)),
\end{align}
where the amplitude exponent and the phase function are given by
\begin{gather*}
r_{l}(z) := 2ay + \left(\theta_{l} + 2\pi\left(\mathbf{1}_{\SS_{<l}}(z) - \mathbf{1}_{\SS_{<0}}(z)\right)\right)\mathrm{Re}\,A, \\
\varphi_{l}(z) := 2ax- \left(\theta_{l} + 2\pi\left(\mathbf{1}_{\SS_{<l}}(z) - \mathbf{1}_{\SS_{<0}}(z)\right)\right)\mathrm{Im}\,A
\end{gather*}
with the set $\SS_{<l}$ defined as
\begin{align*}
\SS_{<l} := \left\{z\in\SS_{F} \ | \ -\frac{2\pi a}{1+a^2} < \mathrm{Re}\,z < -\frac{a\theta_{l}}{1+a^{2}}\right\}, \quad 0\le l \le M-1.
\end{align*} 
Since the function $\tilde w^{*}$ is holomorphic on the strip, from the Cauchy-Riemann equations we obtain
\begin{equation*}
\begin{aligned}
(\tilde{w}_{2})_y(z) = -(\tilde{w}_{1})_x(z) = 2a\sum_{l=0}^{M-1} e^{r_{l}(z)}(A_{2,l}\sin\varphi_{l}(z) - A_{1,l}\cos\varphi_{l}(z)),
\end{aligned}
\end{equation*}
which together with the decay \eqref{v-decay-2} gives
\begin{equation}\label{w2}
\tilde{w}_{2}(z) = \sum_{l=0}^{M-1} e^{r_{l}(z)}(A_{2,l}\sin\varphi_{l}(z) - A_{1,l}\cos\varphi_{l}(z)).
\end{equation}
\subsection{Derivation of boundary conditions.} We now proceed to express the boundary conditions \eqref{eq-bound}–\eqref{eq-39dd} for the function $\tilde w$, given by \eqref{w1} and \eqref{w2}, in terms of the parameters \eqref{log_spirals_parameters} that define the family of logarithmic spirals. To this end, we observe that, for any $0\le l,m \le M-1$, we have
\begin{equation}
\begin{aligned}\label{eq-b-l}
r_{l}^{L}\left(-\frac{a\theta_{m}}{1+a^{2}}+iy\right) & = 2ay + (\theta_{l} - 2\pi\mathbf{1}_{m<l})\mathrm{Re}\,A, \\[2pt] 
\varphi_{l}^{L}\left(-\frac{a\theta_{m}}{1+a^{2}}+iy\right) & = (2\pi\mathbf{1}_{m<l} + \theta_{m} - \theta_{l})\mathrm{Im}\,A,
\end{aligned}
\end{equation}
and, for any $0\le l\le M-1$ and $1\le m \le M$, the following relations hold
\begin{equation}
\begin{aligned}\label{eq-b-r}
r_{l}^{R}\left(-\frac{a\theta_{m}}{1+a^{2}}+iy\right) & = 2ay + (\theta_{l} - 2\pi\mathbf{1}_{m\le l})\mathrm{Re}\,A, \\[2pt]
\varphi_{l}^{R}\left(-\frac{a\theta_{m}}{1+a^{2}}+iy\right) & = (2\pi\mathbf{1}_{m\le l} + \theta_{m} - \theta_{l})\mathrm{Im}\,A.
\end{aligned}
\end{equation}
Then, for $z\in\ell_{0}$, the following equality holds
\begin{equation}
\begin{aligned}\label{b-1}
\tilde w_{1}(z) & = A_{2,0}e^{2ay} + \sum_{l=1}^{M-1} A_{1,l}\sin\left((2\pi-\theta_{l})\,\mathrm{Im}\,A\right)e^{2ay}e^{(\theta_{l}-2\pi)\mathrm{Re}\,A} \\
& \quad + \sum_{l=1}^{M-1} A_{2,l}\cos\left((2\pi-\theta_{l})\,\mathrm{Im}\,A\right)e^{2ay}e^{(\theta_{l}-2\pi)\mathrm{Re}\,A},
\end{aligned}
\end{equation}
whereas, for $z\in\ell_{M}$, we have
\begin{equation}\label{b-2}
\begin{aligned}
\tilde w_{1}(z) & = \sum_{l=0}^{M-1}A_{1,l}\sin\left((2\pi-\theta_{l})\mathrm{Im}\,A\right)e^{2ay}e^{\theta_{l}\mathrm{Re}\,A} \\
&\quad + \sum_{l=0}^{M-1}A_{2,l}\cos\left((2\pi-\theta_{l})\mathrm{Im}\,A\right)e^{2ay}e^{\theta_{l}\mathrm{Re}\,A}.
\end{aligned}
\end{equation}
From \eqref{b-1} and \eqref{b-2} it follows that the condition \eqref{eq-36dd} can be written as
$$\begin{aligned}
\mu & = A_{2,0} + \sum_{l=1}^{M-1}A_{1,l}\sin\left((2\pi - \theta_{l})\mathrm{Im}\,A\right)e^{(\theta_{l}-2\pi)\mathrm{Re}\,A} \\
& \quad + \sum_{l=1}^{M-1}A_{2,l}\cos\left((2\pi - \theta_{l})\mathrm{Im}\,A\right)e^{(\theta_{l}-2\pi)\mathrm{Re}\,A},
\end{aligned}\leqno{(B_{1})}$$
while the second equation of condition \eqref{eq-36dd} takes the form
$$\begin{aligned}
\mu e^{2\pi \mathrm{Re}\,A} & = \sum_{l=0}^{M-1}A_{1,l}\sin((2\pi - \theta_{l})\mathrm{Im}\,A)e^{\theta_{l}\mathrm{Re}\,A} \\
& \quad + \sum_{l=0}^{M-1}A_{2,l}\cos((2\pi - \theta_{l})\mathrm{Im}\,A)e^{\theta_{l}\mathrm{Re}\,A}.
\end{aligned}\leqno{(B_{2})}$$
Observe that given $1\le m \le M-1$ and $z\in\ell_{m}$, we have 
\begin{equation}\label{form-w-1}
\hspace{-5pt}\begin{aligned}
\tilde w_{1}^{L}(z) & = \sum_{l=0}^{M-1}A_{1,l}\sin\left((2\pi\mathbf{1}_{m<l} + \theta_{m} - \theta_{l})\mathrm{Im}\,A \right)\!e^{2ay + (\theta_{l} - 2\pi\mathbf{1}_{m<l})\mathrm{Re}\,A} \\
& \ \ + \sum_{l=0}^{M-1}A_{2,l}\cos\left((2\pi\mathbf{1}_{m<l} + \theta_{m} - \theta_{l})\mathrm{Im}\,A\right)\!e^{2ay + (\theta_{l} - 2\pi\mathbf{1}_{m<l})\mathrm{Re}\,A}
\end{aligned}
\end{equation}
and 
\begin{equation}\label{form-w-2}
\hspace{-9pt}\begin{aligned}
\tilde w_{1}^{R}(z) & =\sum_{l=0}^{M-1}A_{1,l}\sin\left((2\pi\mathbf{1}_{m\le l} + \theta_{m} - \theta_{l})\mathrm{Im}\,A\right)\!e^{2ay + (\theta_{l} - 2\pi\mathbf{1}_{m\le l})\mathrm{Re}\,A} \\
& \ \ + \sum_{l=0}^{M-1}A_{2,l}\cos\left((2\pi\mathbf{1}_{m\le l} + \theta_{m} - \theta_{l})\mathrm{Im}\,A\right)\!e^{2ay + (\theta_{l} - 2\pi\mathbf{1}_{m\le l})\mathrm{Re}\,A}.
\end{aligned}
\end{equation}
By subtracting the equations \eqref{form-w-1} and \eqref{form-w-2}, we see that the first equality of the condition \eqref{eq-bound} is equivalent to 
$$A_{2,m} = \left(A_{1,m}\sin\left(2\pi\,\mathrm{Im}\,A\right) + A_{2,m}\cos\left(2\pi\,\mathrm{Im}\,A\right)\right)e^{-2\pi\mathrm{Re}\,A} \leqno{(B_{3})} $$
for $1\le m \le M-1$. Moreover, by the formula \eqref{form-w-1}, we see that the second equality of \eqref{eq-bound} can be written as
$$\begin{aligned}
\mu e^{\theta_{m}\mathrm{Re}\,A} & = \sum_{l=0}^{M-1}A_{1,l}\sin((2\pi\mathbf{1}_{m < l}+\theta_{m}-\theta_{l})\mathrm{Im}\,A)e^{(\theta_{l} - 2\pi\mathbf{1}_{m<l})\mathrm{Re}\,A} \\
& \ \ + \sum_{l=0}^{M-1}A_{2,l}\cos((2\pi\mathbf{1}_{m<l} + \theta_{m} - \theta_{l})\mathrm{Im}\,A)e^{(\theta_{l} - 2\pi\mathbf{1}_{m<l})\mathrm{Re}\,A}
\end{aligned}\leqno{(B_{4})}$$
for $1\le m\le M-1$. We now apply the boundary conditions to the function $\tilde w_{2}$. Note that, according to \eqref{w2} and \eqref{eq-b-l}, for $z\in\ell_{0}$, the following holds
\begin{align*}
\tilde{w}_{2}(z) & = -A_{1,0}e^{2ay} + \sum_{l=1}^{M-1} A_{2,l}\sin\left((2\pi-\theta_{l})\,\mathrm{Im}\,A\right)e^{2ay}e^{(\theta_{l}-2\pi)\mathrm{Re}\,A} \\
& \quad - \sum_{l=1}^{M-1} A_{1,l}\cos\left((2\pi-\theta_{l})\,\mathrm{Im}\,A\right)e^{2ay}e^{(\theta_{l}-2\pi)\mathrm{Re}\,A},
\end{align*}
while, by \eqref{w2} and \eqref{eq-b-r}, for $z\in\ell_{M}$, the following holds
\begin{equation*}
\begin{aligned}
\tilde w_{2}(z) & = \sum_{l=0}^{M-1}(A_{2,l}\sin((2\pi-\theta_{l})\mathrm{Im}\,A)-A_{1,l}\cos((2\pi-\theta_{l})\mathrm{Im}\,A))e^{2ay}e^{\theta_{l}\mathrm{Re}\,A}.
\end{aligned}
\end{equation*}
In particular, the condition \eqref{eq-39dd} can be expressed as
\begin{equation*}
\begin{aligned}
2ag_{0} e^{2ay} & = \tilde{w}_{2}\left(-\frac{2a\pi}{1+a^2} + i\left(y + \frac{2\pi}{1+a^2}\right)\right) - \tilde{w}_{2}(iy) \\
& = \sum_{l=0}^{M-1}A_{2,l}\sin\left((2\pi-\theta_{l})\mathrm{Im}\,A\right)e^{2ay}e^{(\theta_{l}-2\pi)\mathrm{Re}\,A} \\
&\quad - \sum_{l=0}^{M-1}A_{1,l}\cos\left((2\pi-\theta_{l})\mathrm{Im}\,A\right)e^{2ay}e^{(\theta_{l}-2\pi)\mathrm{Re}\,A} \\
& \quad +A_{1,0}e^{2ay} - \sum_{l=1}^{M-1} A_{2,l}\sin\left((2\pi-\theta_{l})\,\mathrm{Im}\,A\right)e^{2ay}e^{(\theta_{l}-2\pi)\mathrm{Re}\,A} \\
& \quad + \sum_{l=1}^{M-1} A_{1,l}\cos\left((2\pi-\theta_{l})\,\mathrm{Im}\,A\right)e^{2ay}e^{(\theta_{l}-2\pi)\mathrm{Re}\,A},
\end{aligned}
\end{equation*}
which after cancellation of the common terms, leads to 
$$2ag_{0}= A_{2,0}\sin\left(2\pi\mathrm{Im}\,A\right)e^{-2\pi\mathrm{Re}\,A} + A_{1,0}\left(1-\cos\left(2\pi\mathrm{Im}\,A\right)e^{-2\pi\mathrm{Re}\,A}\right). \leqno{(B_{5})}$$
Furthermore, for any $z\in\ell_{m}$, where $1\le m \le M-1$, we have 
\begin{align*}
\tilde w_{2}^{L}(z) & = \sum_{l=0}^{M-1}A_{2,l}\sin\left((2\pi\mathbf{1}_{m<l} + \theta_{m} - \theta_{l})\,\mathrm{Im}\,A \right)e^{2ay + (\theta_{l} - 2\pi\mathbf{1}_{m<l})\mathrm{Re}\,A} \\
& \ \ - \sum_{l=0}^{M-1}A_{1,l}\cos\left((2\pi\mathbf{1}_{m<l} + \theta_{m} - \theta_{l})\,\mathrm{Im}\,A\right)e^{2ay + (\theta_{l} - 2\pi\mathbf{1}_{m<l})\mathrm{Re}\,A}
\end{align*}
and 
\begin{align*}
\tilde w_{2}^{R}(z) & = \sum_{l=0}^{M-1}A_{2,l}\sin\left((2\pi\mathbf{1}_{m\le l} + \theta_{m} - \theta_{l})\,\mathrm{Im}\,A\right)e^{2ay + (\theta_{l} - 2\pi\mathbf{1}_{m\le l})\mathrm{Re}\,A} \\
& \ \ - \sum_{l=0}^{M-1}A_{1,l}\cos\left((2\pi\mathbf{1}_{m\le l} + \theta_{m} - \theta_{l})\,\mathrm{Im}\,A\right)e^{2ay + (\theta_{l} - 2\pi\mathbf{1}_{m\le l})\mathrm{Re}\,A}.
\end{align*}
This implies that the condition \eqref{jump}, takes the form 
$$2ag_{m} = A_{2,m}\sin(2\pi\mathrm{Im}\,A)e^{-2\pi\mathrm{Re}\,A} + A_{1,m}\left(1 - \cos(2\pi\mathrm{Im}\,A)e^{-2\pi\mathrm{Re}\,A}\right)\leqno{(B_{6})}$$ 
for $1\le m \le M-1$.
\subsection{Derivation of the coefficients $A_{m}$.} \label{sec-am} In what follows, we use the boundary conditions $(B_{1})$\,--\,$(B_{6})$ to determine the coefficients $A_{m}$. For this purpose, we apply the condition $(B_{1})$ to the equality $(B_{2})$, in order to obtain
\begin{equation}\label{eq-z-zero}
A_{2,0} = (A_{1,0}\sin\left(2\pi\mathrm{Im}\,A\right) + A_{2,0}\cos\left(2\pi\mathrm{Im}\,A\right))e^{-2\pi\mathrm{Re}\,A}.
\end{equation}
Adding the equation $(B_{3})$ (respectively, \eqref{eq-z-zero}) multiplied by the imaginary unit to $(B_{6})$ (respectively, $(B_{5})$) yields 
\begin{align*}
2ag_{m} & = A_{1,m}-iA_{2,m} + (A_{2,m}+iA_{1,m})\sin\left(2\pi\mathrm{Im}\,A\right)e^{-2\pi\mathrm{Re}\,A} \\
& \quad - (A_{1,m}-iA_{2,m})\cos(2\pi\mathrm{Im}\,A)e^{-2\pi\mathrm{Re}\,A} \\
& = (A_{1,m}-iA_{2,m})\left(1 - \left(\cos(2\pi\mathrm{Im}\,A) - i \sin\left(2\pi\mathrm{Im}\,A\right)\right)e^{-2\pi\mathrm{Re}\,A} \right) \\
& = (A_{1,m}-iA_{2,m})\left(1 - e^{-2\pi A}\right),
\end{align*}
which implies that 
\begin{align}\label{form-a}
-A_{1,m} + i A_{2,m} = \frac{2a g_{m} e^{2\pi A}}{1-e^{2\pi A}}, \quad 0\le m \le M-1.
\end{align}
In particular, we have
\begin{align}\label{from-a1}
A_{1,m} = \frac{2ag_{m}(e^{2\pi\mathrm{Re}\,A} - \cos(2\pi\mathrm{Im}\,A))}{e^{2\pi\mathrm{Re}\,A} + e^{-2\pi\mathrm{Re}\,A} - 2\cos(2\pi\mathrm{Im}\,A)}
\end{align}
and 
\begin{align}\label{from-a2}
A_{2,m} = \frac{2ag_{m}\sin(2\pi\mathrm{Im}\,A)}{e^{2\pi\mathrm{Re}\,A} + e^{-2\pi\mathrm{Re}\,A} - 2\cos(2\pi\mathrm{Im}\,A)}.
\end{align}
\subsection{Formula for the velocity field.} We proceed to the main result of this section, stated in the following proposition.
\begin{prop}\label{prop-form}
Let $w$ be a two-dimensional weakly divergence-free velocity field satisfying the conditions \eqref{decay-m} and \eqref{curl-form}. If the discrete system \eqref{eq-disc2} is satisfied, then the corresponding velocity field $\tilde{w}=\tilde{w}_1+i\tilde{w}_2$ in the strip $\SS_{F}$ is uniquely determined by the component functions \eqref{w1}, \eqref{w2} with the corresponding coefficients given by \eqref{from-a1}, \eqref{from-a2}.
\end{prop}
\begin{proof}
Let us recall that from \cite[Theorem 1.3]{CKO} it follows that the system formed by the real parts of the equations \eqref{def_of_Amk} is equivalent to the velocity matching condition \eqref{vel_match-m}. Hence the uniqueness of the velocity field $w$ is a consequence of the Theorem \ref{th-spiral-unique}. It is enough to show that the coefficients \eqref{from-a1} and \eqref{from-a2} satisfy the conditions $(B_{1})-(B_{6})$. To this end, we take the imaginary part of the discrete system \eqref{eq-disc2} and obtain 
\begin{align}\label{eq-disc2-imag}
\sum_{k=0}^{M-1} \mathrm{Im}(\mathcal{B}_{mk})g_{k} = -\frac{\mu}{a},
\end{align}
where we define $\mathcal{B}_{mk}:= \mathcal{A}_{mk} \sinh^{-1}(\pi A)$. To obtain the precise formulas for the coefficients of the discrete system we observe that, for $k=m$, we have  
\begin{align*}
\mathcal{B}_{mk} =\frac{\cosh(\pi A)}{\sinh(\pi A)} = \frac{e^{2\pi\mathrm{Re}\,A} - e^{-2\pi\mathrm{Re}\,A} - 2i\sin(2\pi\mathrm{Im}\,A)}{e^{2\pi\mathrm{Re}\,A} + e^{-2\pi\mathrm{Re}\,A} - 2\cos(2\pi\mathrm{Im}\,A)}, 
\end{align*} 
which implies that 
\begin{align}\label{im-b-1}
\mathrm{Im}\,\mathcal{B}_{mk} = - 2 D_{a}^{-1}\sin(2\pi\mathrm{Im}\,A),
\end{align}
where we use the notation $$D_{a}:= e^{2\pi\mathrm{Re}\,A} + e^{-2\pi\mathrm{Re}\,A} - 2\cos(2\pi\mathrm{Im}\,A).$$
Furthermore, for $k<m$, the following holds  
\begin{align*}
\mathcal{B}_{mk} = \frac{e^{(\theta_{k}-\theta_{m} + \pi) A}}{\sinh(\pi A)}  = \frac{2e^{(\theta_{k}-\theta_{m})A}\left(e^{2\pi\mathrm{Re}\,A}- e^{2\pi i\mathrm{Im}\,A}\right)}{e^{2\pi\mathrm{Re}\,A} + e^{-2\pi\mathrm{Re}\,A} - 2\cos(2\pi\mathrm{Im}\,A)} 
\end{align*}
and consequently 
\begin{equation}
\hspace{-6pt}\begin{aligned}
\mathrm{Im}\,\mathcal{B}_{mk} & \!=\! 2e^{(\theta_{k}-\theta_{m})\mathrm{Re}\,A}\mathrm{Im}\left(\frac{e^{2\pi\mathrm{Re}\,A}e^{i(\theta_{k}-\theta_{m})\mathrm{Im}\,A} - e^{i(\theta_{k}-\theta_{m}+2\pi)\mathrm{Im}\,A}}{e^{2\pi\mathrm{Re}\,A} + e^{-2\pi\mathrm{Re}\,A} - 2\cos(2\pi\mathrm{Im}\,A)} \right) \\
& = 2D_{a}^{-1}e^{(\theta_{k}-\theta_{m}+2\pi)\mathrm{Re}\,A}\sin((\theta_{k}-\theta_{m})\mathrm{Im}\,A)  \\
& \quad - 2D_{a}^{-1}e^{(\theta_{k}-\theta_{m})\mathrm{Re}\,A}\sin((\theta_{k}-\theta_{m}+2\pi)\mathrm{Im}\,A).
\end{aligned}
\end{equation}
Similarly, for $m<k$, we obtain
\begin{align*}
\mathcal{B}_{mk} = \frac{e^{(\theta_{k}-\theta_{m} - \pi) A}}{\sinh(\pi A)} = \frac{2e^{(\theta_{k}-\theta_{m})A}\left(e^{-2\pi i\mathrm{Im}\,A}-e^{-2\pi\mathrm{Re}\,A}\right)}{e^{2\pi\mathrm{Re}\,A} + e^{-2\pi\mathrm{Re}\,A} - 2\cos(2\pi\mathrm{Im}\,A)} 
\end{align*}
and hence
\begin{equation}\label{im-b-3}
\hspace{-7pt}\begin{aligned}
\mathrm{Im}\,\mathcal{B}_{mk} & = 2e^{(\theta_{k}-\theta_{m})\mathrm{Re}A}\mathrm{Im}\!\left(\!\frac{e^{i(\theta_{k}-\theta_{m}-2\pi)\mathrm{Im}A} - e^{-2\pi\mathrm{Re}A}e^{i(\theta_{k}-\theta_{m})\mathrm{Im}A}}{e^{2\pi\mathrm{Re}\,A} + e^{-2\pi\mathrm{Re}\,A} - 2\cos(2\pi\mathrm{Im}\,A)} \right) \\
& = 2D_{a}^{-1}e^{(\theta_{k}-\theta_{m})\mathrm{Re}\,A}\sin((\theta_{k}-\theta_{m}-2\pi)\mathrm{Im}\,A)  \\
& \quad - 2D_{a}^{-1}e^{(\theta_{k}-\theta_{m}-2\pi)\mathrm{Re}\,A}\sin((\theta_{k}-\theta_{m})\mathrm{Im}\,A).
\end{aligned}
\end{equation}
Observe that the coefficients $A_{1,m}$ and $A_{2,m}$, given by formulas \eqref{form-a} and \eqref{from-a1}, are chosen in Section \ref{sec-am} to satisfy the conditions $(B_{3})$, $(B_{5})$ and $(B_{6})$. To verify that the coefficients meet condition $(B_{1})$, we write 
\begin{align*}
I_{0}\!&:=\! A_{2,0} \!+\! \sum_{l=1}^{M-1}\!(A_{1,l}\sin((2\pi - \theta_{l})\mathrm{Im}\,A)\!+\!A_{2,l}\cos((2\pi - \theta_{l})\mathrm{Im}\,A))e^{(\theta_{l}-2\pi)\mathrm{Re}\,A} \\
& = 2ag_{0}D_{a}^{-1}\sin(2\pi\mathrm{Im}\,A) + D_{a}^{-1}\sum_{l=1}^{M-1} 2ag_{l}\sin((2\pi-\theta_{l})\mathrm{Im}\,A)e^{\theta_{l}\mathrm{Re}\,A}  \\
&\quad - D_{a}^{-1}\sum_{l=1}^{M-1} 2ag_{l}\cos(2\pi\mathrm{Im}\,A)\sin((2\pi-\theta_{l})\mathrm{Im}\,A)e^{(\theta_{l}-2\pi)\mathrm{Re}\,A} \\
&\quad + D_{a}^{-1}\sum_{l=1}^{M-1} 2ag_{l}\sin(2\pi\mathrm{Im}\,A)\cos((2\pi - \theta_{l})\mathrm{Im}\,A)e^{(\theta_{l}-2\pi)\mathrm{Re}\,A}.
\end{align*}
By applying trigonometric formulas and \eqref{im-b-1}\,--\,\eqref{im-b-3}, we calculate further
\begin{align*}
I_{0}& := 2ag_{0}D_{a}^{-1}\sin(2\pi\mathrm{Im}\,A) + D_{a}^{-1}\sum_{l=1}^{M-1}2ag_{l}\sin((2\pi - \theta_{l})\mathrm{Im}\,A)e^{\theta_{l}\mathrm{Re}\,A} \\
& \quad +  D_{a}^{-1}\sum_{l=1}^{M-1}2ag_{l}\sin(\theta_{l}\mathrm{Im} \,A)e^{(\theta_{l}-2\pi)\mathrm{Re}\,A}  = -a\sum_{k=0}^{M-1} \mathrm{Im}(\mathcal{B}_{0k})g_{k} = \mu,
\end{align*}
where in the last equality we used the first equation of the system \eqref{eq-disc2-imag}. Hence the condition $(B_{1})$ follows. Since the coefficients \eqref{from-a1} and \eqref{from-a2} are chosen such that the equality \eqref{eq-z-zero} holds, it follows that condition $(B_{2})$ is also satisfied. To check that the remaining condition $(B_{4})$ is satisfied for all $1\le m \le M-1$, we compute
\begin{align*}
I_{m}& := D_{a}\sum_{l=0}^{M-1}A_{1,l}\sin((2\pi\mathbf{1}_{m < l}+\theta_{m}-\theta_{l})\mathrm{Im}\,A)e^{(\theta_{l} - 2\pi\mathbf{1}_{m<l})\mathrm{Re}\,A} \\
& \quad + D_{a}\sum_{l=0}^{M-1}A_{2,l}\cos((2\pi\mathbf{1}_{m<l} + \theta_{m} - \theta_{l})\mathrm{Im}\,A)e^{(\theta_{l} - 2\pi\mathbf{1}_{m<l})\mathrm{Re}\,A} \\
&  = \sum_{l=0}^{M-1}2ag_{m}e^{2\pi\mathrm{Re}\,A}\sin((2\pi\mathbf{1}_{m < l}+\theta_{m}-\theta_{l})\mathrm{Im}\,A)e^{(\theta_{l} - 2\pi\mathbf{1}_{m<l})\mathrm{Re}\,A} \\
& \quad - \sum_{l=0}^{M-1}2ag_{m}\cos(2\pi\mathrm{Im}\,A)\sin((2\pi\mathbf{1}_{m < l}+\theta_{m}-\theta_{l})\mathrm{Im}\,A)e^{(\theta_{l} - 2\pi\mathbf{1}_{m<l})\mathrm{Re}\,A} \\
& \quad + \sum_{l=0}^{M-1}2ag_{m}\sin(2\pi\mathrm{Im}\,A)\cos((2\pi\mathbf{1}_{m<l} + \theta_{m} - \theta_{l})\mathrm{Im}\,A)e^{(\theta_{l} - 2\pi\mathbf{1}_{m<l})\mathrm{Re}\,A},
\end{align*}
which, after applying trigonometric formulas and \eqref{im-b-1}\,--\,\eqref{im-b-3}, yields
\begin{align*}
I_{m}& = \sum_{l=0}^{M-1}2ag_{m}\sin((2\pi\mathbf{1}_{m < l}+\theta_{m}-\theta_{l})\mathrm{Im}\,A)e^{(\theta_{l} + 2\pi\mathbf{1}_{l\le m}))\mathrm{Re}\,A} \\
& \quad + \sum_{l=0}^{M-1}2ag_{m}\sin((2\pi\mathbf{1}_{l\le m} - \theta_{m} + \theta_{l})\mathrm{Im}\,A)e^{(\theta_{l} - 2\pi\mathbf{1}_{m<l})\mathrm{Re}\,A} \\
& = -aD_{a}e^{\theta_{m}\mathrm{Re}\,A} \sum_{k=0}^{M-1} \mathrm{Im}(\mathcal{B}_{mk})g_{k} = \mu D_{a}e^{\theta_{m}\mathrm{Re}\,A},
\end{align*}
where in the last equality we use the $m$-th equation of the system \eqref{eq-disc2-imag}. This shows that condition $(B_{4})$ is satisfied  and completes the proof of the proposition.
\end{proof}

\section{Velocity field determined by a family of spirals}\label{sec-velocity-m}
As a continuation of the preceding section, we now apply the inverse mapping of the mapping $f$, as established in Lemma \ref{mapping}, on the complex potential $\tilde\Phi$ associated with the vector field $\tilde{w}$, which was obtained in Proposition \ref{prop-form}. Observe that the potential is given by $\tilde\Phi = \tilde \Phi_{1} + i \tilde\Phi_{2}$, where 
\begin{align*}
\tilde\Phi_{1}(z) := \sum_{l=0}^{M-1} e^{r_{l}(z)} \left(-\frac{A_{1,l}}{2a}\cos\varphi_{l}(z) + \frac{A_{2,l}}{2a}\sin\varphi_{l}(z)\right)
\end{align*}
and
\begin{align*}
\tilde\Phi_{2}(z) := \sum_{l=0}^{M-1} e^{r_{l}(z)} \left(\frac{A_{1,l}}{2a}\sin\varphi_{l}(z) + \frac{A_{2,l}}{2a}\cos\varphi_{l}(z)\right),
\end{align*}
which together with \eqref{form-a} implies that 
\begin{equation}\label{eq-bba}
\begin{aligned}
\tilde\Phi(z) & = \sum_{l=0}^{M-1} \left(-\frac{A_{1,l}}{2a} + i\frac{A_{2,l}}{2a} \right) e^{r_{l}(z)} e^{-i\varphi_{l}(z)}\\
& = \sum_{l=0}^{M-1} \left(-\frac{A_{1,l}}{2a} + i\frac{A_{2,l}}{2a}\right)e^{-2aiz}e^{\theta_{l}A}e^{2\pi\left(\mathbf{1}_{\SS_{<l}}(z)-\mathbf{1}_{\SS_{<0}}(z)\right)A} \\
& = \sum_{l=0}^{M-1} \frac{g_{l} e^{2\pi A}}{1-e^{2\pi A}}e^{-2aiz}e^{\theta_{l}A}e^{2\pi\left(\mathbf{1}_{\SS_{<l}}(z)-\mathbf{1}_{\SS_{<0}}(z)\right)A}
\end{aligned}
\end{equation}
for $z\in \SS_{F}$. In what follows, we will need the lemma below.
\begin{lem}\label{lem-winding}
For any $0\le k \le M-1$, we have
\begin{align*}
J(r,\theta, k) - J(r,\theta) + 1 = \mathbf{1}_{\SS_{<k}}(f^{-1}(z)), \quad z = re^{i\theta}\in\RR_{F},
\end{align*}
where 
\begin{align*}
J (r,\theta ,k ) := \min \left\{ j\in \Z  \ | \  a(2\pi j + \theta_k - \theta ) + \ln r >0 \right\}
\end{align*}
and $J(r,\theta)$ is given by the formula \eqref{def-j}.
\end{lem}
\begin{proof}
Given $z = re^{i\theta}\in\RR_{F}$, we use the formula for the inverse of the function $f$ from Lemma \ref{mapping}, to obtain $f^{-1}(z) = x + yi\in \SS_{F}$, where 
\begin{align*}
x = \frac{\ln r - a\theta + 2a\pi(J(r,\theta) - 1)}{1 + a^{2}}, \quad  y=\frac{\theta + a\ln r - 2\pi(J(r,\theta) - 1) }{1 + a^{2}}.
\end{align*}
Then $r=e^{x+ay}$ and $\theta= y-ax + 2\pi(J(r,\theta) - 1)$, which implies that
\begin{align*}
J(r, \theta ,k ) & = \min \left\{ j\in \Z  \ | \  a(2\pi j + \theta_k -  \theta ) + \ln  r >0 \right\} \\
& = \min \left\{ j\in \Z  \ | \  a(2\pi(j-J(r,\theta) + 1) + \theta_k - y+ax) + x+ay >0 \right\} \\
& = \min \left\{ j\in \Z  \ | \  2\pi(j-J(r,\theta) + 1) + \theta_k > - \frac{(1+a^{2})}{a}x \right\} \\
& = \min \left\{ j\in \Z  \ | \  2\pi j + \theta_k > - \frac{(1+a^{2})}{a}x \right\} + J(r,\theta) - 1\\
& = \mathbf{1}_{\SS_{<k}}(x+yi) + J(r,\theta) - 1 \\
& = \mathbf{1}_{\SS_{<k}}(f^{-1}(z)) + J(r,\theta) - 1.
\end{align*}
Thus the proof of the lemma is completed.
\end{proof}
Observe that the formula for the inverse function from Lemma \ref{mapping} gives
\begin{equation}\label{eq-fun-1}
-2aif^{-1}(z) = iA(\ln r + i(\theta - 2\pi J(r,\theta))) - 2\pi A, \quad z=re^{i\theta}\in\RR_{F}.
\end{equation}
Applying \eqref{eq-fun-1} and Lemma \ref{lem-winding} to the equality \eqref{eq-bba}, we obtain
\begin{align*}
\tilde\Phi(f^{-1}(z)) &= \sum_{l=0}^{M-1} \frac{g_{l} e^{2\pi A}}{1-e^{2\pi A}} e^{-2aif^{-1}(z)}e^{\theta_{l}A}e^{2\pi\left(\mathbf{1}_{\SS_{<l}}(f^{-1}(z)) - \mathbf{1}_{\SS_{<0}}(f^{-1}(z))\right)A} \\
& = \sum_{l=0}^{M-1} \frac{g_{l}e^{\theta_{l}A}}{1-e^{2\pi A}} e^{iA(\ln r + i(\theta - 2\pi J(r,\theta)))} e^{2\pi J(r,\theta,l)A}e^{-2\pi J(r,\theta)A} \\
& = \sum_{l=0}^{M-1} \frac{g_{l}e^{\theta_{l}A}}{1-e^{2\pi A}} e^{iA(\ln r + i(\theta - 2\pi J(r,\theta,l)))}
\end{align*}
which coincides exactly with the potential given in \cite[eq. (1.31)]{CKO} for $M\ge 2$.

Hence, we obtain the following result. 
\begin{theorem}\label{th_wiele_spiral}
Let $w$ be a two-dimensional weakly divergence-free velocity field satisfying the decay condition \eqref{decay-m} such that 
\[
\curl w = \sum_{k=0}^{M-1}\gamma_{k}\, \delta_{\Sigma_{k}}
\]
in the sense of distributions. If the spiral parameters satisfy the system of equations \eqref{eq-disc2}, 
then the potential and the velocity field are expressed by 
\begin{equation*}
\Phi(z) = \sum_{l=0}^{M-1} \frac{g_{l}e^{\theta_{l}A}}{1-e^{2\pi A}} e^{iA(\ln r + i(\theta - 2\pi J(r,\theta,l)))}
\end{equation*}
and
\begin{equation*}
w(z) =  e^{i\theta } \sum_{l=0}^{M-1} \frac{2ag_{l} }{r(a-i)} \left(  r^{\frac{2a}{a+i} }e^{A  (\theta_l -\theta ) } \frac{e^{2\pi J (r,\theta ,l)A }}{1-e^{2\pi A}} \right)^*
\end{equation*}
for $z = re^{i\theta} \in\RR_{F}$.
\end{theorem}

\vspace{0.5cm}
\noindent
{\bf Acknowledgment.} Parts of the article were completed during TC's visits to Nicolaus Copernicus University in Toru\'n and PK's visits to IMPAN. Both authors are grateful for the support provided by the {\em Small Meetings} program of IMPAN and the IDUB visiting program at NCU.

\parindent = 0 pt
\end{document}